\documentclass[11pt]{article}
\usepackage{amsfonts}
\usepackage{mathrsfs}

\usepackage[latin1]{inputenc}
\usepackage{amsmath,amssymb}
\usepackage{latexsym}
\usepackage[active]{srcltx}
\usepackage[
bookmarks=true,         %generates bookmarks for all entries in the "Table of Contents"
bookmarksnumbered=true, %includes chapter-/header-/section-/subsection-/... -
colorlinks=true, pdfstartview=FitV, linkcolor=blue, citecolor=blue,
urlcolor=blue]{hyperref}

\newtheorem{theorem}{Theorem}[section]

\newtheorem{lemma}[theorem]{Lemma}
\newtheorem{proposition}[theorem]{Proposition}

\newtheorem{remark}[theorem]{Remark}
\numberwithin{equation}{section}
\def\R{{\mathbb R}}
\def\E{{{\mathbb E}\,}}
\def\P{{\mathbb P}}

\def\sgn{{\mathop {{\rm sgn\, }}}}

\def\square{{\vcenter{\vbox{\hrule height.3pt
        \hbox{\vrule width.3pt height5pt \kern5pt
           \vrule width.3pt}
        \hrule height.3pt}}}}

\def\tlint{{- \kern-0.85em \int \kern-0.2em}}
\def\dlint{{- \kern-1.05em \int \kern-0.4em}}

\def \eref#1{\hbox{(\ref{#1})}}

\def \eref#1{\hbox{(\ref{#1})}}

\newenvironment{proof}[1][Proof]{\noindent\textit{#1.} }{\hfill \rule{0.5em}{0.5em}}

\begin{document}

\title{Limit theorems for functionals of long memory linear processes with infinite variance}
\date{\today}
\author{Hui Liu, Yudan Xiong and Fangjun Xu\thanks{This work is partially supported by the National Natural Science Foundation of China (Grant No. 12371156 and Grant No. 11871219).} \\
}
\maketitle
\begin{abstract} Let $X=\{X_n: n\in\mathbb{N}\}$ be a long memory linear process in which the coefficients are regularly varying and innovations are independent and identically distributed and belong to the domain of attraction of an $\alpha$-stable law with $\alpha\in (0, 2)$. Then, for any integrable and square integrable function $K$ on $\mathbb{R}$, under certain mild conditions, we establish the asymptotic behavior of the partial sum process
\[
\left\{\sum\limits_{n=1}^{[Nt]}\big[K(X_n)-\E K(X_n)\big]:\; t\geq 0\right\}
\] 
as $N$ tends to infinity, where $[Nt]$ is the integer part of $Nt$ for $t\geq 0$.
\noindent


\end{abstract}

\section{Introduction}

Let $X=\{X_n: n\in\mathbb{N}\}$ be a linear process defined by
\begin{align} \label{lp}
X_n=\sum^{\infty}_{i=1} a_i\varepsilon_{n-i}, 
\end{align}
where the innovations $\varepsilon_i$ are independent and identical distributed (i.i.d.) random variables belonging to the domain of attraction of an $\alpha$-stable law with $0<\alpha<2$, $\varepsilon_1$ is centered for $\alpha>1$ and is symmetric for $\alpha=1$, $a_i\sim i^{-\beta} \ell(i)$ with $\alpha\beta>1$ and $\ell$ being a slowly varying function at $\infty$.  Here $\sim$ indicates that the ratio of both sides tends to $1$ as $i$ tends to infinity. So $a_i$ is a regularly varying function at $\infty$ with index $-\beta$. For the innovation $\varepsilon_1$, by Theorem 2.6.1 in \cite{IL}, there exist nonnegative constants $\sigma_1$ and $\sigma_2$ such that
\begin{align} \label{h}
\mathbb{P}(\varepsilon_1\leq -x)&=(\sigma_1+o(1))x^{-\alpha} h(x)\quad \text{and}\quad \mathbb{P}(\varepsilon_1>x)=(\sigma_2 +o(1))x^{-\alpha} h(x)
\end{align}
as $x$ tends to infinity, where $\sigma_1, \sigma_2\geq 0$, $\sigma_1+\sigma_2>0$ and $h(x)$ is a positive slowly varying function at $\infty$. It is easy to see that $\sum\limits^{\infty}_{i=1} |a_i|^{\alpha}h(|a_i|^{-1})<\infty$. Then, according to \cite{A} (or Proposition 5.4 in \cite{BJL}), the linear process $X=\{X_n: n\in\mathbb{N}\}$ defined in \eref{lp} is a.s. convergent. Clearly, the linear process $X=\{X_n: n\in\mathbb{N}\}$ has infinite variance. 

There is not yet a complete agreement on the definition of short and long memory for linear processes with infinite variance. A new definition of short and long memory for general linear processes was proposed in Definition 2.1 of \cite{SSX}. This definition relies on the coefficients $a_i$ and the behavior of the characteristic function of the innovation around the origin. It is consistent with existing ones and works very well for linear processes with innovations in the domain of attraction of $\alpha$-stable law because of Theorem 2.6.5 in \cite{IL} for $\alpha\neq 1$ and Theorem 2 in \cite{AD} for $\alpha=1$. According to Definition 2.1 in \cite{SSX}, the linear process $X=\{X_n: n\in\mathbb{N}\}$ defined in \eref{lp} has short memory if $\alpha\beta>2$ and long memory if $\alpha\beta\in (1,2)$.

Limit theorems for functionals of linear processes with finite variance have been well studied. For linear processes with infinite variance, much attention is paid to the asymptotic behavior of the partial sum process
\begin{align} \label{ps}
\bigg\{ S_{[Nt]}:=\sum^{[Nt]}_{n=1}\big[K(X_n)-\mathbb{E}K(X_n)\big] :\, t\geq 0\bigg\}
\end{align}
as $N$ tends to infinity, where $X=\{X_n: n\in\mathbb{N}\}$ is the linear process defined in \eref{lp} and $K$ is a proper real-valued measurable function. For the short memory case $\alpha\beta>2$, if the innovation $\varepsilon_1$ is a symmetric $\alpha$-stable random variable and $\sum^{\infty}_{i=1}|a_i|^{\alpha/2}<\infty$ (without the restriction $a_i\sim i^{-\beta} \ell(i)$), then, under certain conditions, the weak convergence of $N^{-\frac{1}{2}}S_N$ to a normal distribution was established in \cite{Hsing}. This result was extended to two-sided linear processes with  innovations in the domain of attraction of $\alpha$-stable law and less conditions  in \cite{PT} and to the multidimensional case in \cite{W}. For the long memory case where  $1<\alpha\beta<2$ and $\beta<1$, if the innovations are in the normal domain of attraction of an $\alpha$-stable law (see the definition on page 92 in \cite{IL}) and $a_i\sim c_0 i^{-\beta}$, then under some additional assumptions, the weak convergence of $N^{\beta-\frac{1}{\alpha}-1}S_N$  to an $\alpha$-stable distribution was established in \cite{KS}. Under a strong condition (see Condition 1 in \cite{W}), Wu extended the result in \cite{KS} to the case when innovations are in the domain of attraction of $\alpha$-stable law and $a_i\sim i^{-\beta}\ell(i)$ in \cite{W}. For the long memory case where  $1<\alpha\beta<2$ and $\beta>1$, if the innovations are in the normal domain of attraction of an $\alpha$-stable law and $a_i\sim c_0 i^{-\beta}$, then, under some additional assumptions, the convergence of finite-dimensional distributions of $\{N^{-\frac{1}{\alpha\beta}}S_{[Nt]}:\, t\in [0,1]\}$ to an $\alpha\beta$-stable process was obtained in \cite{S} for the case $\alpha\in(1,2)$ and in \cite{Honda} for the case $\alpha\in (0,1)$. 

In this paper, we will develop a new methodology to establish the asymptotic behavior of the partial sum process (\ref{ps}) under mild conditions. Before stating our results, we first make the following three assumptions:
\begin{enumerate}
\item[({\bf A1})] $K$ is a real-valued integrable and square integrable function on $\mathbb{\R}$, 
\item[({\bf A2})] There exist strictly positive constants $c$ and $\delta$ such that the characteristic function $\phi_{\varepsilon}(u)$ of the innovation $\varepsilon_1$ satisfies $|\phi_{\varepsilon}(u)|\leq \frac{c}{1+|u|^{\delta}}$ for all $u\in\mathbb{R}$,
\item[({\bf A3})]  the slowly varying function $\ell: (0,\infty)\to (0,\infty)$ has the Karamata representation 
\[
\ell(x)=\sigma e^{\int^x_1 \frac{\eta(t)}{t}}dt, \quad x>0,
\] 
where $\sigma>0$ and $\lim\limits_{t\to\infty}\eta(t)=0$.
\end{enumerate}
The assumption ({\bf A2}) and $a_i\sim i^{-\beta}\ell(i)$ imply that  (i) the linear process $X=\{X_n: n\in\mathbb{N}\}$ defined in \eref{lp} has a bounded and differentiable probability density function $f(x)$ and (ii) the characteristic function $\phi(u)$ of $X_n$ decays faster than any polynomial rate to $0$ as $|u|$ tends to infinity. Moreover, ({\bf A2}) implies that there exists $m\in\mathbb{N}$ such that $|\phi_{\varepsilon}(u)|^m$ is less than a constant multiple of $\frac{1}{1+|u|^4}$.  The assumption ({\bf A3}) is only used in Theorems \ref{thm2} and \ref{thm3}. We make assumption ({\bf A3}) to simplify the notation. According to the Karamata representation of a slowly varying function, $\ell$ has the expression 
$\ell(x)=\sigma(x) e^{\int^x_1 \frac{\eta(t)}{t}}dt$ where $\lim\limits_{x\to\infty}\sigma(x)=\sigma\neq 0$ and $\lim\limits_{t\to\infty}\eta(t)=0$. Recall that $a_i\sim i^{-\beta}\ell(i)$ and $\sim$ indicates that the ratio of both sides tends to $1$ as $i$ tends to infinity. So we can assume that $\ell$ satisfies the assumption ({\bf A3}). Moreover, there is another advantage of making the assumption ({\bf A3}). If assumption ({\bf A3}) holds, then it is easy to see that there exists a constant $A>0$ such that the regularly varying function 
\begin{align} \label{lbx}
\ell_{\beta}(x):=x^{\frac{1}{\beta}}\ell^{\frac{1}{\beta}}(x^{\frac{1}{\beta}})
\end{align}
is continuous and strictly increasing on the interval $(A,\infty)$. This is needed in the definition of the normalizing factors in Theorems \ref{thm2} and \ref{thm3}.

We next introduce some notations. For any $\beta\in (\frac{1}{\alpha}, \frac{1}{\alpha}+1)$ and $t\geq 0$, define 
\[
Z^{\alpha, \beta}_t=\int^t_{-\infty} [(t-s)^{1-\beta}-(-s)^{1-\beta}1_{\{s<0\}}] dZ^{\alpha}_s,
\] 
where $Z^{\alpha}=\{Z^{\alpha}_t: t\geq 0\}$ is an $\alpha$-stable process whose characteristic function has the form
\[
\mathbb{E} e^{\iota u Z^{\alpha}_t}= \begin{cases} \exp\Big(-t|u|^{\alpha}(1-\iota \frac{\sigma_2-\sigma_1}{\sigma_1+\sigma_2}\tan \frac{\pi \alpha}{2} \sgn u)\Big)
	 & \text { for } \alpha \neq 1 \\ \\
	 \exp\Big(-t|u|^{\alpha}\Big)
	 & \text { for } \alpha=1\end{cases}
\]
with $\iota=\sqrt{-1}$ and $\sigma_i$ $(i=1,2)$ being the same constants as in (\ref{h}).

According to \cite{A}, the process $Z^{\alpha,\beta}=\{Z^{\alpha,\beta}_t: t\geq 0\}$ is $\alpha$-stable, self-similar with index $\frac{1}{\alpha}+1-\beta$ and has stationary increments; the trajectories of $Z^{\alpha,\beta}$ are continuous when $\beta<1$ and discontinuous when $\beta\geq 1$. For the slowly varying function $h$ in (\ref{h}), by Proposition 1 (iv) in \cite{A}, up to asymptotic equivalence, there exists a unique slowly varying function $h_{\alpha}$ such that $h(N^{\frac{1}{\alpha}}h^{\frac{1}{\alpha}}_{\alpha}(N))\sim h_{\alpha}(N)$ as $N\to\infty$.  Given $A>0$. Suppose that $g$ is a nondecreasing function on $(A,\infty)$ with range $(0,\infty)$.  With the convention that the infimum of an empty set is $\infty$. We define the inverse $g^{\leftarrow}: (0,\infty)\to (A,\infty)$ of $g$ as 
\[
g^{\leftarrow}(x)=\inf\{s>A: g(s)\geq x\}
\]
and let $K_{\infty}(x)=\mathbb{E} K(X_1+x)$ for each $x\in\R$.

Now we can address the main results in this paper.

\begin{theorem} \label{thm1} Under the assumptions  ({\bf A1}) and ({\bf A2}), if $1<\alpha<2$ and $\frac{1}{\alpha}<\beta<1$, then, as $N$ tends to infinity,
\begin{enumerate}
\item[(i)]
	The finite-dimensional distributions of $\left\{N^{\beta-\frac{1}{\alpha}-1}\ell^{-1}(N)h_{\alpha}^{-1/\alpha}(N)S_{[Nt]}: t\geq 0\right\}$ converge to those of $\big\{ \widetilde{c} \, Z^{\alpha, \beta}_t: t\geq 0\big\}$;
\item[(ii)]  $\left\{N^{\beta-\frac{1}{\alpha}-1}\ell^{-1}(N)h_{\alpha}^{-1/\alpha}(N)S_{[Nt]}: t\in [0,1]\right\}$ converges in law to $\big\{ \widetilde{c} \, Z^{\alpha, \beta}_t: t\in [0,1]\big\}$ in $D([0,1])$ endowed with the Skorohod topology,
\end{enumerate}
where 
\[
\widetilde{c}=\frac{1}{1-\beta}(\sigma_1+\sigma_2)^{\frac{1}{\alpha}}\Big(\Gamma(\alpha-1) \cos (\frac{\pi \alpha}{2})\Big)^{\frac{1}{\alpha}}\Big(-\int_{\mathbb{R}}K(x)df(x)\Big).
\]
\end{theorem}

\begin{theorem} \label{thm2} Under the assumptions  ({\bf A1}), ({\bf A2}) and ({\bf A3}), if $\alpha\beta\in (1,2)$, $\beta>1$ and 
\begin{align} \label{ell}
\lim\limits_{x\to\infty} \frac{\ell(x\ell^{\frac{1}{\beta}}(x))}{\ell(x)}=1,
\end{align}
then the finite-dimensional distributions of $\Bigg\{\frac{1}{\Big(\frac{(\ell^{\leftarrow}_{\beta})^{\alpha}}{h\circ (\ell^{\leftarrow}_{\beta})}\Big)^{\leftarrow}(N)} S_{[Nt]}: t\geq 0 \Bigg\}$ converge to those of  $\Big\{(\gamma_2+\gamma_1)^{\frac{1}{\alpha\beta}}\big(\overline{c}\, t^{\frac{1}{\alpha\beta}}+\mathcal{Z}^{\alpha\beta}_t \big): t\geq 0\Big\}$ as $N$ tends to infinity, where  
\begin{align*}
\gamma_2&=\sigma_2 (C^+_K)^{\alpha\beta} 1_{\{C^+_K>0\}}+\sigma_1 (C^-_K)^{\alpha\beta} 1_{\{C^-_K>0\}}\;\; \text{and}\;\; \gamma_1=\sigma_2 |C^+_K|^{\alpha\beta} 1_{\{C^+_K<0\}}+\sigma_1 |C^-_K|^{\alpha\beta} 1_{\{C^-_K<0\}}
\end{align*}
with 
\[
C^{\pm}_K=\int^{\infty}_0 (K_{\infty}(\pm t^{-\beta})-K_{\infty}(0))dt,
\] 
\begin{align*}
\overline{c}=\frac{\gamma_2-\gamma_1}{\gamma_2+\gamma_1}\left[\frac{1}{\alpha\beta-1}+\alpha\beta\left(\int^{\infty}_1\frac{x^{-\alpha\beta}}{x^2+1}dx-\int^1_0\frac{x^{2-\alpha\beta}}{x^2+1}dx\right)+\int^{\infty}_0  \frac{x^{2-\alpha\beta}(x^{2}+3)}{(x^2+1)^2} dx \right] 
\end{align*}
and $\mathcal{Z}^{\alpha\beta}=\{\mathcal{Z}^{\alpha\beta}_t:\, t\geq 0\}$ is an $\alpha\beta$-stable process whose characteristics function has the form
\begin{align*} %\label{Z}
\mathbb{E} e^{\iota u \mathcal{Z}^{\alpha\beta}_t} \nonumber
&=\exp\Bigg(-t\left(\int^{\infty}_0  \frac{\sin x}{x^{\alpha\beta}} dx\right)|u|^{\alpha\beta}\left[1-\iota \frac{\gamma_2-\gamma_1}{\gamma_2+\gamma_1} \sgn(u)\tan(\frac{\pi\alpha\beta}{2})\right]\Bigg).
\end{align*}
\end{theorem}

\begin{theorem} \label{thm3} Under the assumptions ({\bf A1}), ({\bf A2}) and ({\bf A3}), if $1<\alpha<2$, $\frac{1}{\alpha}<\beta\leq 1$, $\int_{\mathbb{R}}K(x)df(x)=0$ and 
\begin{align*} \label{ell}
\lim\limits_{x\to\infty} \frac{\ell(x\ell^{\frac{1}{\beta}}(x))}{\ell(x)}=1,
\end{align*}
then the finite-dimensional distributions of $\Bigg\{\frac{1}{\Big(\frac{(\ell^{\leftarrow}_{\beta})^{\alpha}}{h\circ (\ell^{\leftarrow}_{\beta})}\Big)^{\leftarrow}(N)} S_{[Nt]}: t\geq 0 \Bigg\}$ converge to those of $\Big\{ (\gamma_2+\gamma_1)^{\frac{1}{\alpha\beta}}\big(\overline{c}\, t^{\frac{1}{\alpha\beta}}+\mathcal{Z}^{\alpha\beta}_t\big): t\geq 0\Big\}$ as $N$ tends to infinity, where $\gamma_2$, $\gamma_1$,  $\overline{c}$ and $\mathcal{Z}^{\alpha\beta}=\{\mathcal{Z}^{\alpha\beta}_t:\, t\geq 0\}$ are defined as in Theorem \ref{thm2}.
\end{theorem}

\begin{remark}
(i)  Recall that the characteristic function $\phi(u)$ of $X_n$ decays faster than any polynomial rate to $0$ as $|u|$ tends to infinity. This implies that the derivative of the probability density function $f(x)$ of the linear process $X=\{X_n: n\in\mathbb{N}\}$ is bounded. Hence, $\int_{\mathbb{R}}K(x)df(x)$ in Theorem \ref{thm1} is well-defined.

(ii) Under the assumption ({\bf A3}), we see that 
\[
\ell(x)=\sigma e^{\int^x_1 \frac{\eta(t)}{t} dt}, \quad x>0,
\] 
where $\sigma>0$ and $\lim\limits_{t\to\infty}\eta(t)=0$. If there exists a constant $c_1>0$ such that 
\begin{align} \label{ellchar}
|\eta(t)|\leq \frac{c_1}{(\ln t)^a}
\end{align}
for some $a\in (\frac{1}{2},1)$ and all $t>1$, then, for all $x>1$,
\begin{align} \label{ellbound}
\ln \sigma-\frac{c_1}{1-a}(\ln x)^{1-a}\leq \ln \ell(x)\leq \ln \sigma+\frac{c_1}{1-a}(\ln x)^{1-a}
\end{align}
and
\begin{align*}
\limsup_{x\to\infty} \Big|\int^{x\ell^{\frac{1}{\beta}}(x)}_x \frac{\eta(t)}{t}dt\Big|
&\leq \frac{c_1}{1-a} \limsup_{x\to\infty}  \Big|(\ln x+\frac{1}{\beta} \ln \ell(x))^{1-a}-(\ln x)^{1-a}\Big|\\
&\leq \frac{c_1}{1-a} \limsup_{x\to\infty} \, (\ln x)^{1-a} \Big|(1+\frac{ \ln \ell(x)}{\beta\ln x})^{1-a}-1\Big|\\
&=0,
\end{align*}
where we use mean value theorem, (\ref{ellbound}) and $a\in (\frac{1}{2}, 1)$ in the last equality. Therefore, 
\[
\lim\limits_{x\to\infty} \frac{\ell(x\ell^{\frac{1}{\beta}}(x))}{\ell(x)}=e^{\lim\limits_{x\to\infty} \int^{x\ell^{\frac{1}{\beta}}(x)}_x \frac{\eta(t)}{t}dt}=1
\]
and thus (\ref{ellchar}) with $a\in (\frac{1}{2}, 1)$ is a sufficient condition for (\ref{ell}).  Moreover, for $\ell(x)=\sigma e^{\int^x_1 \frac{\eta(t)}{t} dt}$, it is easy to see that
$\eta(x)=x \frac{d}{dx} (\ln \ell(x))$. Hence, (\ref{ellchar}) could be replaced by $|\frac{d}{dt} (\ln \ell(t))|\leq \frac{c_1}{t(\ln t)^a}$.

If $\eta(t)=(\ln t)^{-a}$ for some $a\in (0,\frac{1}{2}]$ and all $t>1$, then
\begin{align} \label{counterexample}
\ell(x)=\sigma e^{\int^x_1 \frac{\eta(t)}{t} dt}=\sigma e^{\frac{1}{1-a}(\ln x)^{1-a}}
\end{align}
and
\begin{align*}
\lim_{x\to\infty} \int^{x\ell^{\frac{1}{\beta}}(x)}_x \frac{\eta(t)}{t}dt
&=\frac{1}{1-a} \lim_{x\to\infty}  \Big[\big(\ln x+\frac{\ln\sigma}{\beta}+\frac{1}{\beta(1-a)}(\ln x)^{1-a}\big)^{1-a}-(\ln x)^{1-a}\Big]\\
&=\frac{1}{1-a} \lim_{x\to\infty} (\ln x)^{1-a} \Big[\big(1+\frac{\ln\sigma}{\beta\ln x}+\frac{1}{\beta(1-a)} (\ln x)^{-a}\big)^{1-a}-1\Big]\\
&=\left\{
        \begin{array}{ll}
             \frac{2}{\beta} & \quad \text{for }\;\; a=\frac{1}{2},\\ \\
            \infty & \quad \text{for}\;\; a\in(0,\frac{1}{2}),
        \end{array}
    \right.
\end{align*}
where we use mean value theorem and $a\in (0,\frac{1}{2}]$ in the last equality. Therefore, for $a\in (0,\frac{1}{2}]$, (\ref{counterexample}) gives counterexamples to $(\ref{ell})$.

(iii) By Proposition 2.6 in \cite{R}, it is easy to see that $\left(\frac{(\ell^{\leftarrow}_{\beta})^{\alpha}}{h\circ (\ell^{\leftarrow}_{\beta})}\right)^{\leftarrow}(x)$ is a regularly varying function at $\infty$ with index $\frac{1}{\alpha\beta}$.

\end{remark}

We next compare our results with those in \cite{KS, S, Honda}. Roughly speaking, Theorem \ref{thm1} extends the result in \cite{KS} by allowing the innovations to be in the general domain of attraction of $\alpha$-stable law and the coefficients of the underlying linear process to possess the slowly varying function $\ell$. It also improves the weak convergence for random variables in \cite{KS} to weak convergence in the space $D[0,1]$ endowed with the Skorohod topology. Theorem \ref{thm2} generalizes both results in \cite{S} and \cite{Honda} by including the case $\alpha=1$ and allowing the innovations to be in the general domain of attraction of $\alpha$-stable law and coefficients to possess the slowly varying function $\ell$. Theorem \ref{thm3} is an improvement of Theorem \ref{thm1} in the case $\int_{\mathbb{R}}K(x)df(x)=0$ and answers the open question mentioned in Remark 2.2 of \cite{KS}. Moreover, our methodology is different from those in \cite{KS, S, Honda}. 

It is interesting to note that Theorem \ref{thm3} is valid for the case $\beta=1$ which is not included in Theorem \ref{thm1}. We believe that Theorem \ref{thm1} is also true for the case $\beta=1$ with a proper modification in the normalizing factor.  Moreover, our results show that the asymptotic behavior of the partial sum process $\{S_{[Nt]}: t\geq 0\}$ belong to two classes. One is for the case $\beta\in (\frac{1}{\alpha}, 1]$ and the other is for the case $\beta\in (1,\frac{2}{\alpha})$. In particular, if $\int_{\mathbb{R}}K(x)df(x)=0$ in the case $\beta\in (\frac{1}{\alpha}, 1]$, then the asymptotic behavior of the partial sum process $\{S_{[Nt]}: t\geq 0\}$ switches immediately to the one in the case $\beta\in (1,\frac{2}{\alpha})$. Recall the corresponding normalizing factors in these two cases. We see that this kind of switch in the normalizing factors is quite different from the one between the usual first and second order limit laws.

To obtain our results,  a new methodology is developed to study the asymptotic behavior of the partial sum process $\{S_{[Nt]}: t\geq 0\}$ defined in (\ref{ps}). According to the proofs of Theorems \ref{thm1}, \ref{thm2} and \ref{thm3},  the convergence of finite-dimensional distributions and the tightness in the space $D([0,1])$ rely heavily on the asymptotic behavior of the partial sum $S_{[Nt]}$ for each $t>0$.  Note that the normalizing factors in these theorems are regularly varying functions at infinity.  So we only need to consider the asymptotic behavior of the partial sum $S_N$ as $N$ tends to infinity.  Our methodology is based on Fourier transform, orthogonal projection in $L^2(\Omega,\mathcal{F},\mathbb{P})$ and von Bahr-Esseen inequality in \cite{vBE}.  It helps us easily find the main ingredient in the partial sum $S_N$ which will contribute to the limiting distribution, see Proposition \ref{prop}, Lemmas \ref{lemrd}, \ref{lm51} and \ref{lm53} below. It could probably be applied to the study of limit theorems for functionals of general linear processes with heavy-tailed innovations.

Throughout this paper, if not mentioned otherwise, the letter $c$ with or without a subscript, denotes a generic positive finite constant whose exact value is independent of $n$ and may change from line to line. We use $\iota$ to denote the imaginary unit $\sqrt{-1}$. For a complex number $z$, we use $\overline{z}$ and $|z|$ to denote its conjugate and modulus, respectively. For any integrable function $g(x)$, its Fourier transform is defined as $\widehat{g}(u)=\int_{\mathbb{R}}e^{\iota x u}g(x)\mathrm{d}x$. Moreover, we let $\phi(\lambda)$ be the characteristic function of linear process $X=\{X_n:\, n\in\mathbb{N}\}$ and $\phi_{\varepsilon}(\lambda)$ the characteristic function of innovations. That is, $\phi(\lambda)=\mathbb{E}[e^{\iota \lambda X_n}]$ and $\phi_{\varepsilon}(\lambda)=\mathbb{E}[e^{\iota \lambda \varepsilon_{1}}]$.  Finally, $[x]$ denotes the integer part of $x\geq 0$. 

The paper has the following structure. After some preliminaries in Section 2, Section 3 is devoted to the proofs of Theorems \ref{thm1}, \ref{thm2} and \ref{thm3} based on the Fourier transform, orthogonal projection and von Bahr-Esseen inequality.

\section{Preliminaries}

Let $Z$ be a real-valued $\alpha$-stable random variable. Then the characteristic function of $Z$ has the following expression
\[
\mathbb{E}[e^{\iota \lambda Z}]=\exp\Big(\iota \lambda \mu -\sigma^{\alpha}|\lambda|^{\alpha}(1-\iota \eta \sgn(\lambda) \omega(\lambda,\alpha))\Big),
\]
where $0<\alpha \leq 2,\sigma>0,-1 \leq \eta \leq 1,\mu \in \mathbb{R}$, and
\begin{equation*}\label{def-omega}
\omega(\lambda,\alpha)= \begin{cases} \tan(\frac{\pi \alpha}{2})
	 & \text { for } \alpha \neq 1 \\ \\
	 2 \pi^{-1} \ln(|\lambda|)
	 & \text { for } \alpha=1.\end{cases}
\end{equation*}
It is called symmetric if $\eta=\mu=0$ and standard if $\sigma=1$. For more details on stable laws, we refer to \cite{ST}.  A real-valued random variable $\varepsilon$ is said to be in the domain of attraction of an $\alpha$-stable law if there exist i.i.d. random variables $\varepsilon_n$ with the same distribution as $\varepsilon$, real numbers $A_n$, and strictly positive numbers $B_n$ such that 
\[
B^{-1}_n\Big(\sum\limits^n_{i=1}\varepsilon_i-A_n\Big)\xrightarrow{\mathcal{L}} Z
\] 
as $n$ tends infinity, where ``$\xrightarrow{\mathcal{L}}$'' denotes the convergence in distribution (see, e.g., \cite{IL}). The next lemma gives some estimates for the characteristic function of random variables in the domain of attraction of an $\alpha$-stable law. It will play an important role in the proofs of our main results.

% Since the innovations $\varepsilon_i$  belong to the domain of attraction of an $\alpha$-stable law, by Theorem 2.6.5 in \cite{IL}, the characteristic function $\phi_{\varepsilon}$ of $\varepsilon_1$ satisfies 
%\[
%|\phi_{\varepsilon}(\lambda)|=e^{-c_{\alpha} |\lambda|^{\alpha} L(\lambda)(1+o(1))}
%\]
%as $\lambda\to 0$, where $c_{\alpha}$ is a positive constant depending on $\alpha$ and $L(\lambda)$ is a slowly varying function as $\lambda\to 0$.  

\begin{lemma}  \label{lemb}
Suppose that $\varepsilon$ is in the domain of attraction of an $\alpha$-stable law with $\alpha\in (0,2)$, $\mathbb{E}\, \varepsilon=0$ for $\alpha\in (1,2)$, and $\phi_{\varepsilon}(\lambda)$ is the characteristic function of $\varepsilon$. Then, for any $\alpha'\in (0,\alpha)$, there exist positive constants $c_{\alpha',1}$ and $c_{\alpha',2}$ such that
\begin{align} \label{ineq1}
1-|\phi_{\varepsilon}(\lambda)|^2=\E\big|e^{ \iota \lambda\varepsilon}-\phi_{\varepsilon}(\lambda)\big|^2\leq c_{\alpha',1}\left(|\lambda|^{\alpha'}\wedge 1\right)
\end{align}
and 
\begin{align} \label{ineq2}
|1-\phi_{\varepsilon}(\lambda)|\leq c_{\alpha',2}\left(|\lambda|^{\alpha'}\wedge 1\right).
\end{align}
\end{lemma}

\begin{proof}  Note that $1-|\phi_{\varepsilon}(\lambda)|^2\leq 2(1-|\phi_{\varepsilon}(\lambda)|)\leq 2|1-\phi_{\varepsilon}(\lambda)|$.  So it suffices to show the inequality (\ref{ineq2}).  When $\alpha\in (0,1]$,  the inequality (\ref{ineq2}) follows from
\begin{align*}
|1-\phi_{\varepsilon}(\lambda)|=|\mathbb{E}(1-\cos(\lambda \varepsilon))-\iota \mathbb{E} \sin(\lambda \varepsilon) |
&\leq c_1\, ( |\lambda|^{\alpha'} \mathbb{E} |\varepsilon|^{\alpha'})\wedge 1\leq c_2 (|\lambda|^{\alpha'}\wedge 1),
\end{align*}
where in the last inequality we use Theorem 2.6.4 in \cite{IL}. When $\alpha\in (1,2)$, the inequality (\ref{ineq2}) follows from Theorem 2.6.4 in \cite{IL} and Lemma 5 in \cite{vBE}.
\end{proof}

\section{Proofs of main results}

In this section, we first find the main ingredient in the partial sum 
\[
S_N=\sum^{N}_{n=1}\big[K(X_n)-\mathbb{E}K(X_n)\big]
\]
which will contribute to the limiting distribution in the general case, see Proposition \ref{prop} below.  

For each $j\in\mathbb{Z}$, let $\mathcal{F}_{j}$ be the $\sigma$-field generated by $\{\varepsilon_i: i\leq j\}$. By assumptions ({\bf A1}) and ({\bf A2}), $K(X_n)\in L^2(\Omega,\mathcal{F},\mathbb{P})$ for each $n\in\mathbb{N}$.  Hence, by the orthogonal projection in $L^2(\Omega,\mathcal{F},\mathbb{P})$ or martingale convergence theorem, 
\begin{align*}
S_{N}=\sum^N_{n=1} \sum^{\infty}_{j=1}\Big( \mathbb{E}\big[K(X_n)|\mathcal{F}_{n-j}\big]-\mathbb{E}\big[K(X_n)|\mathcal{F}_{n-j-1}\big]\Big).
\end{align*}
Define
\begin{align} \label{tn1}
T_{N}=\sum^N_{n=1} \sum^{\infty}_{j=1}\Big( \mathbb{E}\big[K(X_n+a_j\widetilde{\varepsilon}_{n-j})|\varepsilon_{n-j}\big]-\mathbb{E}\big[K(X_n+a_j\widetilde{\varepsilon}_{n-j})\big]\Big),
\end{align}
where $\{\widetilde{\varepsilon}_i: k\in\mathbb{Z}\}$ is an independent copy of $\{\varepsilon_i: k\in\mathbb{Z}\}$. It is easy to see that the terms in the infinite series 
\[
\sum^{\infty}_{j=1}\Big( \mathbb{E}\big[K(X_n+a_j\widetilde{\varepsilon}_{n-j})|\varepsilon_{n-j}\big]-\mathbb{E}\big[K(X_n+a_j\widetilde{\varepsilon}_{n-j})\big]\Big)
\]
are independent. Note that 
\begin{align*}
 \mathbb{E}\big[K(X_n+a_j\widetilde{\varepsilon}_{n-j})|\varepsilon_{n-j}\big]
 &=\int_{\mathbb{R}} K(x+a_j\varepsilon_{n-j})f(x)dx=\frac{1}{2\pi}\int_{\mathbb{R}} \widehat{K}(u)\phi(-u)e^{-\iota ua_{j}\varepsilon_{n-j}}du,
\end{align*}
where we use Plancherel formula in the last equality. Therefore,
\begin{align*}
&\mathbb{E} \Big|\mathbb{E}\big[K(X_n+a_j\widetilde{\varepsilon}_{n-j})|\varepsilon_{n-j}\big]-\mathbb{E}\big[K(X_n+a_j\widetilde{\varepsilon}_{n-j})\big]\Big|^2\\
&=\frac{1}{4\pi^2} \mathbb{E}\Big|\int_{\mathbb{R}} \widehat{K}(u)\phi(-u)(e^{-\iota ua_{j}\varepsilon_{n-j}}-\phi_{\varepsilon}(-u a_{j}) du \Big|^2\\
&=\frac{1}{4\pi^2} \int_{\mathbb{R}^2} \widehat{K}(u)\widehat{K}(v) \phi(-u) \phi(-v)\mathbb{E}\big[(e^{-\iota ua_{j}\varepsilon_{n-j}}-\phi_{\varepsilon}(-u a_{j})) (e^{-\iota va_{j}\varepsilon_{n-j}}-\phi_{\varepsilon}(-v a_{j}))\big] du\, dv,
\end{align*}
where we use Fubini theorem in the last equality. 

By Cauchy-Schwartz inequality and Lemma \ref{lemb}, we could easily get that 
\[
\mathbb{E} \Big|\mathbb{E}\big[K(X_n+a_j\widetilde{\varepsilon}_{n-j})|\varepsilon_{n-j}\big]-\mathbb{E}\big[K(X_n+a_j\widetilde{\varepsilon}_{n-j})\big]\Big|^2
\]
is less than a constant multiple of $|a_j|^{\alpha'}$ for any $\alpha'\in(0,\alpha)$. Choosing $\alpha'$ close enough to $\alpha$ gives that the infinite series 
\[
\sum^{\infty}_{j=1}\Big( \mathbb{E}\big[K(X_n+a_j\widetilde{\varepsilon}_{n-j})|\varepsilon_{n-j}\big]-\mathbb{E}\big[K(X_n+a_j\widetilde{\varepsilon}_{n-j})\big]\Big)
\]
converges in $L^2$. Recall that the terms in the infinite series are independent. The convergence is also almost surely by L\'{e}vy's equivalence theorem (see, e.g., \cite{dudley}).

Now we obtain the following proposition which gives the main ingredient in the partial sum $S_N$.
\begin{proposition} \label{prop} Suppose that assumptions ({\bf A1}) and ({\bf A2}) hold. If $1<\alpha\beta\leq \frac{3}{2}$, then, for any $\alpha'\in (0,\alpha)$ and $\beta'\in (0,\beta)$ with $\alpha'\beta'>1$, there exists a positive constant $c_{\alpha', \beta{'},1}$ such that
\begin{align*}
\E|S_{N}-T_{N}|^2\leq c_{\alpha', \beta',1}\, N^{4-2\alpha'\beta'}.
\end{align*}
If $\alpha\beta>\frac{3}{2}$, then, for any $\alpha'\in (0,\alpha)$ and $\beta'\in (0,\beta)$ with $\alpha'\beta'>\frac{3}{2}$, there exists a positive constant $c_{\alpha',\beta', 2}$ such that
\begin{align*}
\E|S_{N}-T_{N}|^2\leq c_{\alpha',\beta',2}\,  N.
\end{align*}
\end{proposition}

\begin{proof}
The proof will be done in several steps.

\noindent
{\bf Step 1} By assumption ({\bf A2}), there exists $m_0\in\mathbb{N}$ such that $\prod^{m_0}_{i=1}|\phi_{\varepsilon}(a_iu)|$ is less than a constant multiple of $\frac{1}{1+|u|^4}$. Let 
\[
S_{N,1}=\sum^N_{n=1} \sum^{\infty}_{j=m_0+1}\Big( \mathbb{E}\big[K(X_n)|\mathcal{F}_{n-j}\big]-\mathbb{E}\big[K(X_n)|\mathcal{F}_{n-j-1}\big]\Big).
\]
Then we will show that $\mathbb{E}|S_N-S_{N,1}|^2$ is less than a constant multiple of $N$.  

Observe that
\begin{align*}
S_N-S_{N,1}=\sum^N_{n=1}\sum^{m_0}_{j=1}\left( \mathbb{E}\big[K(X_n)|\mathcal{F}_{n-j}\big]-\mathbb{E}\big[K(X_n)|\mathcal{F}_{n-j-1}\big]\right):=\sum^N_{n=1}\sum^{m_0}_{j=1}\mathcal{P}_{n,n-j}
\end{align*}
and $\mathbb{E}[\mathcal{P}_{n_1,n_1-j_1}\mathcal{P}_{n_2,n_2-j_2}]=0$ if $n_1-j_1\neq n_2-j_2$. 

Hence
\begin{align*}
\mathbb{E}|S_N-S_{N,1}|^2
&=\sum^N_{n=1}\sum^{m_0}_{j=1}\mathbb{E}[\mathcal{P}_{n,n-j}^2]+\sum^N_{n=1}\sum_{1\leq i\neq j\leq m_0}\mathbb{E}[\mathcal{P}_{n,n-j}\mathcal{P}_{n-j+i,n-j}]\\
&\leq \sum^N_{n=1}\mathbb{E}[K^2(X_n)]+\sum^N_{n=1}\sum_{1\leq i\neq j\leq m_0}\left(\mathbb{E}[\mathcal{P}^2_{n,n-j}]+\mathbb{E}[\mathcal{P}^2_{n-j+i,n-j}]\right)\\
&\leq c_1 N.
\end{align*}

\noindent
{\bf Step 2} We estimate $\E|S_{N,1}-T_{N,1}|^2$ where 
\[
T_{N,1}=\sum^N_{n=1} \sum^{\infty}_{j=m_0+1}\Big( \mathbb{E}\big[K(X_n)|\varepsilon_{n-j}\big]-\mathbb{E}\big[K(X_n)\big]\Big).
\]
By assumption ({\bf A1}), the choice of $m_0$ in {\bf Step 1} and Plancherel formula, it is easy to see that
\begin{align*}
S_{N,1}-T_{N,1}
&=\frac{1}{2\pi}\sum^N_{n=1}\sum^{\infty}_{j=m_0+1}\int_{\R} \widehat{K}(u)\prod^{j-1}_{k=1}\phi_{\varepsilon}(-a_k u)(e^{-\iota ua_j\varepsilon_{n-j}}-\phi_{\varepsilon}(-a_ju))(e^{-\iota u \widetilde{X}_{n, j}} -\mathbb{E} e^{-\iota u \widetilde{X}_{n, j}})du\\
&:=\frac{1}{2\pi}\sum^N_{n=1}\sum^{\infty}_{j=m_0+1} A_{n,j},
\end{align*}
where $\widetilde{X}_{n,j}=\sum\limits^{\infty}_{k=j+1} a_k\varepsilon_{n-k}$.

Observe that $\mathbb{E} [A_{n_1, j_1}A_{n_2, j_2}]=0$ if $n_1-j_1\neq n_2-j_2$. Hence, we could easily get 
\begin{align} \label{tn12}
\E|S_{N,1}-T_{N,1}|^2 \nonumber
&=\frac{1}{4\pi^2}\sum^N_{n=1}\sum^N_{l=1}\sum^{\infty}_{j=m_0+1}\int_{\R^2} \widehat{K}(u)\widehat{K}(v)\prod^{j-1}_{k=1}\phi_{\varepsilon}(-a_k u)\prod^{i-1}_{p=1}\phi_{\varepsilon}(-a_p v)1_{\{n-j=l-i, i\geq m_0+1\}}\\ \nonumber
&\qquad\qquad\times \E[(e^{-\iota ua_j\varepsilon_{n-j}}-\phi_{\varepsilon}(-a_ju))(e^{-\iota v a_i\varepsilon_{l-i}}-\phi_{\varepsilon}(-a_iv))] \\ 
&\qquad\qquad\qquad\times\E[(e^{-\iota u \widetilde{X}_{n, j}} -\mathbb{E} e^{-\iota u \widetilde{X}_{n, j}})(e^{-\iota v \widetilde{X}_{l, i}} -\mathbb{E} e^{-\iota u \widetilde{X}_{l, i}})] du\, dv.
\end{align}

By the orthogonal projection in $L^2(\Omega,\mathcal{F},\mathbb{P})$,
\begin{align*}
e^{-\iota u \widetilde{X}_{n, j}} -\mathbb{E} e^{-\iota u \widetilde{X}_{n, j}}
&=\sum^{\infty}_{k=j+1}\Big(\mathbb{E}[e^{-\iota u \widetilde{X}_{n, j}}|\mathcal{F}_{n-k}]-\mathbb{E}[e^{-\iota u \widetilde{X}_{n, j}}|\mathcal{F}_{n-k-1}]\Big)\\
&=\sum^{\infty}_{k=j+1} \Big(\prod^{k-1}_{l=j+1}\phi_{\varepsilon}(-ua_l)\Big)\Big(e^{-\iota ua_k\varepsilon_{n-k}}-\phi_{\varepsilon}(-a_k u)\Big)\prod^{\infty}_{l=k+1} e^{-\iota u a_l\varepsilon_{n-l}}
\end{align*}
with the convention $\prod\limits^{j}_{l=j+1}\phi_{\varepsilon}(-ua_l)=1$.

Hence, for any $\alpha'\in(0,\alpha)$, by Lemma \ref{lemb},
\begin{align} \label{xnj}
\E|e^{-\iota u \widetilde{X}_{n, j}} -\mathbb{E} e^{-\iota u \widetilde{X}_{n, j}}|^2\leq \sum^{\infty}_{k=j+1} \E|e^{-\iota ua_k\varepsilon_{n-k}}-\phi_{\varepsilon}(-a_k u)|^2\leq c_{\alpha', 1} |u|^{\alpha'} \sum^{\infty}_{k=j+1} |a_k|^{\alpha'}.
\end{align}
Now, by Cauchy-Schwartz inequality, Lemma \ref{lemb}, equality (\ref{tn12}), inequality (\ref{xnj}), and the choice of $m_0$ in {\bf Step 1}, 
\begin{align*}
\E|S_{N,1}-T_{N,1}|^2
&\leq c_2 \sum^N_{n=1}\sum^N_{l=1}\sum^{\infty}_{j=m_0+1}   1_{\{n-j=l-i, i\geq m_0+1\}} \Big(|a_j|^{\alpha'}\sum^{\infty}_{k=j+1} |a_k|^{\alpha'}\Big)^{\frac{1}{2}} \Big(|a_i|^{\alpha'}\sum^{\infty}_{p=i+1} |a_p|^{\alpha'}\Big)^{\frac{1}{2}}.
\end{align*}
Recall the expression of coefficients $a_j$. Then, for any $\beta'\in (0,\beta)$ with $\alpha'\beta'>1$, by Proposition 1(i) in \cite{A}, we can obtain
\begin{align*}
\E|S_{N,1}-T_{N,1}|^2
&\leq c_3 \sum^N_{n=1}\sum^N_{l=1}\sum^{\infty}_{j=m_0+1} 1_{\{n-j=l-i, i\geq m_0+1\}} j^{\frac{1}{2}-\alpha'\beta'} i^{\frac{1}{2}-\alpha'\beta'}\\
&\leq c_4\bigg[ \sum^N_{n=1}\sum^{\infty}_{j=m_0+1} j^{1-2\alpha'\beta'}+\sum^N_{n=1}\sum^N_{l=n+1}\sum^{\infty}_{j=m_0+1} j^{\frac{1}{2}-\alpha'\beta'}(l-n+j)^{\frac{1}{2}-\alpha'\beta'}\bigg]\\
&\leq c_5\bigg[ N+\sum^N_{n=1}\sum^N_{l=n+1} (l-n)^{2-2\alpha'\beta'}\int^{\infty}_{\frac{1}{l-n}} x^{\frac{1}{2}-\alpha'\beta'}(1+x)^{\frac{1}{2}-\alpha'\beta'} dx \bigg].
\end{align*}

By simple calculation, for $l=n+1,\cdots, N$,
\begin{align*}
\int^{\infty}_{\frac{1}{l-n}} x^{\frac{1}{2}-\alpha'\beta'}(1+x)^{\frac{1}{2}-\alpha'\beta'} dx
\leq c_6 \left\{
        \begin{array}{ll}
             1 & \quad \text{if}\;\; 1<\alpha\beta\leq \frac{3}{2} \; \text{and}\; \alpha'\beta'>1,\\ \\
             (l-n)^{\alpha'\beta'-\frac{3}{2}}  & \quad \text{if}\;\; \alpha'\beta'>\frac{3}{2}.
        \end{array}
    \right.
\end{align*}
Therefore,
\begin{align*}
\E|S_{N,1}-T_{N,1}|^2\leq c_7 \left\{
        \begin{array}{ll}
             N^{4-2\alpha'\beta'} & \quad \text{if}\;\; 1<\alpha\beta\leq \frac{3}{2}\; \text{and}\; \alpha'\beta'>1, \\  \\
             N  & \quad \text{if}\;\; \alpha'\beta'>\frac{3}{2}.
        \end{array}
    \right.
\end{align*}

\noindent
{\bf Step 3} Let 
\[
T_{N,2}=\sum^N_{n=1} \sum^{\infty}_{j=m_0+1}\Big( \mathbb{E}\big[K(X_n+a_j\widetilde{\varepsilon}_{n-j})|\varepsilon_{n-j}\big]-\mathbb{E}\big[K(X_n+a_j\widetilde{\varepsilon}_{n-j})\big]\Big).
\]
We will show that $\E |T_{N,1}-T_{N,2}|^2$ is less than a constant multiple of $N$.  

Note that
\begin{align*}
T_{N,1}-T_{N,2}
&=\frac{1}{2\pi}\sum^N_{n=1} \sum^{\infty}_{j=m_0+1} \int_{\mathbb{R}} \widehat{K}(u) \prod^{j-1}_{k=1}\phi_{\varepsilon}(-a_ku)(1-\phi_{\varepsilon}(-a_j u))(e^{-\iota u a_j \varepsilon_{n-j}}-\phi_{\varepsilon}(-a_j u)) \mathbb{E} e^{-\iota u \widetilde{X}_{n, j}}\, du.
\end{align*}
Recall the choice of $m_0$ in {\bf Step 1}. Then, for any $\alpha'\in(0,\alpha)$ and $\beta'\in (0,\beta)$ with $\alpha'\beta'>1$, by Lemma \ref{lemb} and Proposition 1(i) in \cite{A}, 
\begin{align*}
\E|T_{N,1}-T_{N,2}|^2
&\leq c_{8}\sum^N_{n=1}\sum^N_{l=1}\sum^{\infty}_{j=m_0+1} 1_{\{n-j=l-i, i\geq m_0+1\}} j^{-\frac{3\alpha'\beta'}{2}} i^{-\frac{3\alpha'\beta'}{2}}\\
&\leq c_{9}\bigg(\sum^N_{n=1}\sum^{\infty}_{j=m_0+1} j^{-3\alpha'\beta'}+\sum^N_{n=1}\sum^N_{l=n+1}\sum^{\infty}_{j=m_0+1} j^{-\frac{3\alpha'\beta'}{2}}(l-n+j)^{-\frac{3\alpha'\beta'}{2}}\bigg)\\
&\leq c_{10}\bigg(N+\sum^N_{n=1}\sum^N_{l=n+1} (l-n)^{1-3\alpha'\beta'}\int^{\infty}_{\frac{1}{l-n}} x^{-\frac{3\alpha'\beta'}{2}}(1+x)^{-\frac{3\alpha'\beta'}{2}} dx \bigg)\\
&\leq c_{11}\bigg(N+\sum^N_{n=1}\sum^N_{l=n+1} (l-n)^{-\frac{3\alpha'\beta'}{2}}\bigg)\\
&\leq c_{12}\, N.
\end{align*}

\noindent
{\bf Step 4} We will show that $\mathbb{E}|T_{N}-T_{N,2}|^2$ is less than a constant multiple of $N$. This follows easily from
\begin{align*}
T_{N}-T_{N,2}
&=\sum^N_{n=1} \sum^{m_0}_{j=1}\Big( \mathbb{E}\big[K(X_n+a_j\widetilde{\varepsilon}_{n-j})|\varepsilon_{n-j}\big]-\mathbb{E}\big[K(X_n+a_j\widetilde{\varepsilon}_{n-j})\big]\Big)
\end{align*}
and similar arguments as in {\bf Step 1}.

\noindent 
{\bf Step 5} Combining all estimates in {\bf Step 1}-{\bf Step 4} gives the desired result.
\end{proof}

\bigskip
 
Recall the normalizing factors in Theorems \ref{thm1}, \ref{thm2} and \ref{thm3}.  Choose $\alpha'$ and $\beta'$ close enough to $\alpha$ and $\beta$. Then, with the help of  Proposition 1(i) in \cite{A}, Proposition \ref{prop} implies that the asymptotic behavior of the partial sum $S_N$ is determined by that of $T_{N}$ given in (\ref{tn1}).

In the following subsections, we will start with $T_N$ and give the proofs of Theorems \ref{thm1}, \ref{thm2} and \ref{thm3}, respectively. In each subsection, we first further find the main ingredient in $T_N$ according to the region which $\alpha$ and $\beta$ belong to. Then we will obtain the limit theorem satisfied by the main ingredient and thus give the proof of the corresponding result. 

Recall that $K_{\infty}(x)=\mathbb{E}[K(X_1+x)]$ for each $x\in\mathbb{R}$.  $T_N$ can be rewritten as 
\begin{align} \label{tn2}
T_{N}=\sum^{N}_{n=1} \sum^{\infty}_{j=1} \big(K_{\infty}(a_j \varepsilon_{n-j})-\E K_{\infty}(a_j \varepsilon_{n-j})\big).
\end{align}

\subsection{Proof of Theorem \ref{thm1}}

In this subsection, we give the proof of Theorem \ref{thm1}. Let 
\[
U_N=\left(-\int_{\R} K(x)df(x)\right) \sum\limits^{N}_{n=1} X_n.
\] 
Recall the expression of $T_N$ in (\ref{tn2}) and note that 
\[
-\int_{\R} K(x)df(x)=\frac{1}{2\pi}\int_{\R}\widehat{K}(u) \phi(-u)(-\iota u) du.
\]
Hence
\begin{align*}
T_N-U_N=\frac{1}{2\pi}\sum^{N}_{n=1} \sum^{\infty}_{j=1}  \int_{\R} \widehat{K}(u) \phi(-u) (e^{-\iota u a_j \varepsilon_{n-j}}-\phi_{\varepsilon}(-a_j u)+\iota u a_j \varepsilon_{n-j}) \, du.
\end{align*}

Next we give the following estimates for $T_N-U_N$. 
\begin{lemma} \label{lemrd} Suppose that assumptions ({\bf A1})-({\bf A2}) hold. For any $r\in (1,\alpha\beta)$ with $\alpha\in (1,2)$ and $\beta\in (\frac{1}{\alpha},1]$, there is a positive constant $c_{r}>0$ such that
\[
\E |T_{N}-U_{N}|^{r} \leq c_{r}\, N.
\]
\end{lemma}

\begin{proof}
We first observe that 
\begin{align*}
T_N-U_N=\frac{1}{2\pi}\sum^{N-1}_{j=-\infty} \sum^{N-1}_{n=1\vee (j+1)}  \int_{\R} \widehat{K}(u) \phi(-u) (e^{-\iota u a_{n-j} \varepsilon_{j}}-\phi_{\varepsilon}(-a_{n-j} u)+\iota u a_{n-j} \varepsilon_{j}) du.
\end{align*}

Recall that $\mathbb{E} \varepsilon_j=0$ for $\alpha\in (1,2)$. Then, for any $r\in (1,\alpha\beta)$, by von Bahr-Esseen inequality in \cite{vBE}, 
\begin{align*}
\E |T_{N}-U_{N}|^{r}
\leq \sum^{N-1}_{j=-\infty} \mathbb{E} \Big|\sum^{N-1}_{n=1\vee (j+1)}  \int_{\R} \widehat{K}(u) \phi(-u) (e^{-\iota u a_{n-j} \varepsilon_j}-\phi_{\varepsilon}(-a_{n-j} u)+\iota u a_{n-j} \varepsilon_j) du\Big|^{r}.
\end{align*}

For any $\delta\in (\frac{1}{\beta}, \frac{\alpha}{r})$, by Lemma \ref{lemb} and $|e^{\iota x}-1-\iota x|\leq c_{\delta} |x|^{\delta}$ for all $x\in\R$,
\begin{align*}
|e^{-\iota u a_{n-j} \varepsilon_{j}}-\phi_{\varepsilon}(-a_{n-j} u)+\iota u a_{n-j} \varepsilon_j|
&\leq |e^{-\iota u a_{n-j} \varepsilon_{j}}-1+\iota u a_{n-j} \varepsilon_j|+|1-\phi_{\varepsilon}(-a_{n-j} u)|\\
&\leq c_1 |u|^{\delta}(|a_{n-j} \varepsilon_j|^{\delta}+|a_{n-j}|^{\delta}).
\end{align*}

Since $\alpha\in (1,2)$, $\beta\in (\frac{1}{\alpha},1]$, $\delta\in (\frac{1}{\beta}, \frac{\alpha}{r})$ and $r\in (1,\alpha\beta)$, it is easy to see that 
\[
0<2-\alpha\beta<1+r-\alpha\beta<1+r(1-\beta\delta)<1.
\] 
So there exists $\beta'\in (0,\beta)$ such that $\beta' \delta>1$ and $1+r(1-\beta'\delta)\in(0,1)$. 

Now, using assumptions ({\bf A1})-({\bf A2}), Theorem 2.6.4 in \cite{IL} and Proposition 1(i) in \cite{A},
\begin{align}  \label{sumint}
\E |T_{N}-U_{N}|^{r}  \nonumber
&\leq c_2\, \mathbb{E}(|\varepsilon_{1}|^{r\delta}+1) \sum^{N-1}_{j=-\infty} \Big|\sum^{N-1}_{n=1\vee (j+1)} |a_{n-j}|^{\delta} \Big|^{r} \\  \nonumber
&\leq c_3\, \sum^{N-1}_{j=1} \Big|\sum^{N-1}_{n=j+1} (n-j)^{-\beta'\delta} \Big|^{r}+c_3\, \sum^0_{j=-\infty} \Big|\sum^{N-1}_{n=1} (n-j)^{-\beta'\delta} \Big|^{r}\\  \nonumber
&\leq c_4\, \bigg[N+\sum^{\infty}_{j=1} \Big(j^{1-\beta'\delta}-(N+j)^{1-\beta'\delta} \Big)^{r} \bigg]\\  \nonumber
&\leq c_5\, \bigg[N+\sum^{N}_{j=1} \Big(j^{1-\beta'\delta}-(N+j)^{1-\beta'\delta} \Big)^{r}+N^{1+r(1-\beta'\delta)}\int^{\infty}_1 \Big( x^{1-\beta'\delta}- (1+x)^{1-\beta'\delta} \Big)^{r} dx\bigg]\\ 
&\leq c_6\, N,
\end{align}
where in the last inequality we use $1+r(1-\beta'\delta)\in(0,1)$ and the fact that 
\[
0<x^{1-\beta'\delta}-(1+x)^{1-\beta'\delta}=(\beta'\delta-1)\int^{x+1}_x t^{-\beta'\delta} dt<\int^{x+1}_x t^{-\beta'\delta} dt<x^{-\beta'\delta}
\]
for all $x>0$.
\end{proof}

\medskip
\noindent
{\bf Proof of Theorem \ref{thm1}}:  We divide the proof into two parts.

\noindent
{\bf Part I}: We show the convergence of finite-dimensional distributions. Note that 
\[
N^{\beta-\frac{1}{\alpha}-1}\ell^{-1}(N)h_{\alpha}^{-1/\alpha}(N)
\]
is a regularly varying function at $\infty$ with index $\beta-\frac{1}{\alpha}-1$. For any $t>0$, by Proposition \ref{prop} and Lemma \ref{lemrd}, we could easily obtain that 
\[
N^{\beta-\frac{1}{\alpha}-1}\ell^{-1}(N)h_{\alpha}^{-1/\alpha}(N)(S_{[Nt]}-U_{[Nt]})
\] 
converges in probability to $0$ as $N$ tends to infinity. Then, the desired convergence of finite-dimensional distributions follows easily from  Theorem 1(ii) in \cite{A}.

\noindent
{\bf Part II}: We show that the family
\[
\left\{N^{\beta-\frac{1}{\alpha}-1}\ell^{-1}(N)h_{\alpha}^{-1/\alpha}(N)S_{[Nt]}:\; t\in [0,1]\right\}_{N\geq 1}
\]
is tight in the Skorohod space $D([0,1])$. By Theorem 13.5 in \cite{Billingsley}, it suffices to show that there exist $\theta_1>0$, $\theta_2>0$ and $c>0$ such that
\begin{align} \label{tightness}
\mathbb{P}\bigg(\min\Big\{|S_{[Nt]}-S_{[Nt_1]}|, |S_{[Nt_2]}-S_{[Nt]}|\Big\}\geq \lambda N^{1+\frac{1}{\alpha}-\beta}\ell(N)h_{\alpha}^{1/\alpha}(N) \bigg)\leq \frac{c}{\lambda^{\theta_2}}(t_2-t_1)^{1+\theta_1} 
\end{align}
for all $ 0\leq t_1\leq t\leq t_2\leq 1$, $N\geq 1$ and $\lambda>0$. 

Note that the left hand side of (\ref{tightness}) is equal to $0$ when $0<t_2-t_1<\frac{1}{N}$ and
\begin{align*}
&\mathbb{P}\bigg(\min\Big\{|S_{[Nt]}-S_{[Nt_1]}|, |S_{[Nt_2]}-S_{[Nt]}|\Big\}\geq \lambda N^{1+\frac{1}{\alpha}-\beta}\ell(N)h_{\alpha}^{1/\alpha}(N) \bigg)\\
&\leq \min\left\{\mathbb{P}\Big(|S_{[Nt]}-S_{[Nt_1]}|\geq \lambda N^{1+\frac{1}{\alpha}-\beta}\ell(N)h_{\alpha}^{1/\alpha}(N) \Big),\mathbb{P}\Big(|S_{[Nt_2]}-S_{[Nt]}|\geq \lambda N^{1+\frac{1}{\alpha}-\beta}\ell^{1}(N)h_{\alpha}^{1/\alpha}(N) \Big)\right\}\\
&=\min\left\{\mathbb{P}\Big(|S_{[Nt]-[Nt_1]}|\geq \lambda N^{1+\frac{1}{\alpha}-\beta}\ell(N)h_{\alpha}^{1/\alpha}(N) \Big),\mathbb{P}\Big(|S_{[Nt_2]-[Nt]}|\geq \lambda N^{1+\frac{1}{\alpha}-\beta}\ell(N)h_{\alpha}^{1/\alpha}(N) \Big)\right\},
\end{align*}
where in the last equality we use the fact that $\{K(X_n)-\mathbb{E}K(X_n):\, n\geq 1\}$ is a stationary sequence. 

Therefore, we only need to show that there exist $\theta_1>0$, $\theta_2>0$ and $c>0$ such that
\begin{align} \label{tightness1}
\mathbb{P}\Big(\big|S_{[Nt_2]-[Nt_1]}\big|\geq \lambda N^{1+\frac{1}{\alpha}-\beta}\ell(N)h_{\alpha}^{1/\alpha}(N) \Big)\leq \frac{c}{\lambda^{\theta_2}}(t_2-t_1)^{1+\theta_1} 
\end{align}
for all $ 0\leq t_1\leq t_2\leq 1$, $N\geq 1$,  $t_2-t_1\geq \frac{1}{N}$ and $\lambda>0$.  

Note that 
\[
S_{[Nt_2]-[Nt_1]}=S_{[Nt_2]-[Nt_1]}-T_{[Nt_2]-[Nt_1]}+T_{[Nt_2]-[Nt_1]}-U_{[Nt_2]-[Nt_1]}+U_{[Nt_2]-[Nt_1]}
\]
and thus
\begin{align*}
&\mathbb{P}\Big(\big|S_{[Nt_2]-[Nt_1]}\big|\geq \lambda N^{1+\frac{1}{\alpha}-\beta}\ell(N)h_{\alpha}^{1/\alpha}(N) \Big)\\
&\qquad\leq \mathbb{P}\Big(\big|S_{[Nt_2]-[Nt_1]}-T_{[Nt_2]-[Nt_1]}\big|\geq \lambda N^{1+\frac{1}{\alpha}-\beta}\ell(N)h_{\alpha}^{1/\alpha}(N)/3 \Big)\\
&\qquad\qquad+\mathbb{P}\Big(\big|T_{[Nt_2]-[Nt_1]}-U_{[Nt_2]-[Nt_1]}\big|\geq \lambda N^{1+\frac{1}{\alpha}-\beta}\ell(N)h_{\alpha}^{1/\alpha}(N)/3 \Big)\\
&\qquad\qquad\qquad+\mathbb{P}\Big(\big|U_{[Nt_2]-[Nt_1]}\big|\geq \lambda N^{1+\frac{1}{\alpha}-\beta}\ell(N)h_{\alpha}^{1/\alpha}(N)/3 \Big).
\end{align*}

In the sequel, we will show that the three terms on the right hand side of the above inequality can have an upper bound as in (\ref{tightness1}).

By Proposition \ref{prop} and the Chebyshev inequality, for any $\alpha'\in(0,\alpha)$, $\beta'\in(0,\beta)$ with $\alpha'\beta'>1$, there exists a positive constant $c_{\alpha',\beta'}$ such that
\begin{align*}
&\mathbb{P}\Big(\big|S_{[Nt_2]-[Nt_1]}-T_{[Nt_2]-[Nt_1]}\big|\geq \lambda N^{1+\frac{1}{\alpha}-\beta}\ell(N)h_{\alpha}^{1/\alpha}(N)/3 \Big)\\
&\leq  c_{\alpha',\beta'}\begin{cases}
             \frac{\big([Nt_2]-[Nt_1]\big)^{4-2\alpha'\beta'}}{\lambda^2 N^{2+\frac{2}{\alpha}-2\beta}\ell^{2}(N)h_{\alpha}^{2/\alpha}(N)} & \quad \text{if}\;\; 1<\alpha\beta\leq \frac{3}{2} \; \text{and}\; \alpha'\beta'>1,\\ \\
           \frac{[Nt_2]-[Nt_1]}{\lambda^2 N^{2+\frac{2}{\alpha}-2\beta}\ell^{2}(N)h_{\alpha}^{2/\alpha}(N)} & \quad \text{if}\;\; \alpha\beta>\frac{3}{2} \; \text{and}\; \alpha'\beta'>\frac{3}{2}.
        \end{cases}
\end{align*} 
Note that $t_2-t_1\geq \frac{1}{N}$. So
\[
0\leq [Nt_2]-[Nt_1]<Nt_2+1-(Nt_1-1)\leq 3N(t_2-t_1)
\]
and thus
\begin{align*}
&\mathbb{P}\Big(\big|S_{[Nt_2]-[Nt_1]}-T_{[Nt_2]-[Nt_1]}\big|\geq \lambda N^{1+\frac{1}{\alpha}-\beta}\ell(N)h_{\alpha}^{1/\alpha}(N)/3 \Big)\\
&\leq  9\, c_{\alpha',\beta'}\begin{cases}
             \frac{N^{2+2\beta-\frac{2}{\alpha}-2\alpha'\beta'}}{\lambda^2 \ell^{2}(N)h_{\alpha}^{2/\alpha}(N)} (t_2-t_1)^{4-2\alpha'\beta'}& \quad \text{if}\;\; 1<\alpha\beta\leq \frac{3}{2} \; \text{and}\; \alpha'\beta'>1,\\ \\
           \frac{N^{2\beta-1-\frac{2}{\alpha}}}{\lambda^2 \ell^{2}(N)h_{\alpha}^{2/\alpha}(N)} (t_2-t_1) & \quad \text{if}\;\; \alpha\beta>\frac{3}{2} \; \text{and}\; \alpha'\beta'>\frac{3}{2}.
        \end{cases}
\end{align*} 
Recall that $1<\alpha<2$ and $\frac{1}{\alpha}<\beta<1$. In the case $1<\alpha\beta\leq \frac{3}{2}$,  
\[
2+2\beta-\frac{2}{\alpha}-2\alpha'\beta'=\frac{2}{\alpha}(\alpha\beta-1)(1-\alpha)+2(\alpha\beta-\alpha'\beta').
\]
Choose $\alpha'$ and $\beta'$ close enough to $\alpha$ and $\beta$, respectively. We can obtain that 
\[
\frac{N^{2+2\beta-\frac{2}{\alpha}-2\alpha'\beta'}}{\ell^{2}(N)h_{\alpha}^{2/\alpha}(N)}
\] 
is uniformly bounded and $4-2\alpha'\beta'>1$. So 
\[
\mathbb{P}\Big(\big|S_{[Nt_2]-[Nt_1]}-T_{[Nt_2]-[Nt_1]}\big|\geq \lambda N^{1+\frac{1}{\alpha}-\beta}\ell(N)h_{\alpha}^{1/\alpha}(N)/3 \Big)
\]
can have an upper bound as in (\ref{tightness1}). In the case $\alpha\beta>\frac{3}{2}$, 
\[
2\beta-1-\frac{2}{\alpha}=\frac{1}{\alpha}(\alpha\beta-\alpha+\alpha\beta-2)<0.
\]
Recall that $t_2-t_1\geq \frac{1}{N}$.  Choose a constant $\theta_1\in (0,1+\frac{2}{\alpha}-2\beta)$. Then $N^{-\theta_1}\leq (t_2-t_1)^{\theta_1}$ and $\frac{N^{2\beta-1-\frac{2}{\alpha}+\theta_1}}{\ell^2(N)h^{2/\alpha}_{\alpha}(N)}$ is uniformly bounded. Hence
\[
\mathbb{P}\Big(\big|S_{[Nt_2]-[Nt_1]}-T_{[Nt_2]-[Nt_1]}\big|\geq \lambda N^{1+\frac{1}{\alpha}-\beta}\ell(N)h_{\alpha}^{1/\alpha}(N)/3 \Big)
\]
can have an upper bound as in (\ref{tightness1}).

By Lemma \ref{lemrd} and similar arguments as above, 
\[
\mathbb{P}\Big(\big|T_{[Nt_2]-[Nt_1]}-U_{[Nt_2]-[Nt_1]}\big|\geq \lambda N^{1+\frac{1}{\alpha}-\beta}\ell(N)h_{\alpha}^{1/\alpha}(N)/3 \Big)
\]
can have an upper bound as in (\ref{tightness1}).

Finally, using similar arguments as in the proof of Theorem 2 in \cite{A}, 
\[
\mathbb{P}\Big(\big|U_{[Nt_2]-[Nt_1]}\big|\geq \lambda N^{1+\frac{1}{\alpha}-\beta}\ell(N)h_{\alpha}^{1/\alpha}(N)/3 \Big)
\]
can have an upper bound as in (\ref{tightness1}). This completes the proof.

\subsection{Proof of Theorem \ref{thm2}}

In this subsection, we give the proof of Theorem \ref{thm2}. Let
\begin{align} \label{taun}
\mathcal{T}_{N}=\sum^{N}_{n=1} \sum^{\infty}_{j=1} \big(K_{\infty}(a_j \varepsilon_{n})-\E K_{\infty}(a_j \varepsilon_{n}) \big).
\end{align}
The convergence of the infinite series in (\ref{taun}) follows easily from (\ref{Ksum}) below. Then we have the following estimates for $T_{N}-\mathcal{T}_{N}$.
\begin{lemma} \label{lm51} Suppose that assumptions ({\bf A1})-({\bf A2}) hold, $\alpha\beta\in (1,2)$ and $\beta>1$.  For any $\alpha'\in (0,\alpha)$, $\beta'\in (1,\beta)$ with $\alpha\in (0,1]$ and $\alpha'\beta'>1$, there is a positive constant $c_{\alpha',\beta',3}>0$ such that
\[
\E |T_{N}-\mathcal{T}_{N}| \leq c_{\alpha',\beta',3}\, N^{2-\alpha'\beta'}.
\]
For any $\alpha'\in (1,\alpha)$, $\beta'\in (1,\beta)$ with $\alpha\in (1,2)$ and $\alpha'\beta'>1$, there is a positive constant $c_{\alpha',\beta',4}>0$ such that
\[
\E |T_{N}-\mathcal{T}_{N}|^{\alpha'} \leq c_{\alpha',\beta',4}\, N^{1+\alpha'(1-\beta')}.
\]

\end{lemma}

\begin{proof} Note that 
\begin{align*}
\mathcal{T}_{N}-T_{N}=R_{N,1}-R_{N,2}
&:=\left(\sum^N_{n=1}\sum^{\infty}_{j=N-n+1}-\sum^0_{n=-\infty}\sum^{N-n}_{j=1-n}\right)\Big(K_{\infty}(a_j\varepsilon_n)-\E K_{\infty}(a_j\varepsilon_n)\Big)
\end{align*}
and
\begin{align*}
K_{\infty}(a_j\varepsilon_n)-\E K_{\infty}(a_j\varepsilon_n)
&=\frac{1}{2\pi}\int_{\R} \widehat{K}(u) \phi(-u)(e^{-\iota u a_j \varepsilon_{n}}-\phi_{\varepsilon}(-a_j u)) du.
\end{align*}

Let $K^{\infty}_{j,n}=K_{\infty}(a_j\varepsilon_n)-\E K_{\infty}(a_j\varepsilon_n)$. Then, assumptions ({\bf A1})-({\bf A2}) and Lemma \ref{lemb},
\begin{align} \label{Ksum}
|K^{\infty}_{j,n}|   \nonumber
&\leq \int_{\R} |\widehat{K}(u)| |\phi(-u)||(|e^{-\iota u a_j \varepsilon_{n}}-1|+|1-\phi_{\varepsilon}(-a_j u)|) du\\  \nonumber
&\leq c_{1}\, \int_{\R} |\widehat{K}(u)| |\phi(u)||( |u a_j \varepsilon_{n}|^{\min(\alpha',1)}+|a_j u|^{\alpha'}) du\\ 
&\leq c_{2}\, \Big(|a_j|^{\min(\alpha',1)} |\varepsilon_n|^{\min(\alpha',1)}+|a_j|^{\alpha'}\Big).
\end{align}

By Theorem 2.6.4 in \cite{IL}, $\mathbb{E} |\varepsilon_n|^{\alpha'}<\infty$. We first consider the case $\alpha\in(0,1]$. Let $\delta=\alpha'\beta'$. Clearly, $\delta\in (1,2)$. By Proposition 1 in \cite{A}, we can obtain that
\begin{align*}
\E |R_{N,1}|+\E |R_{N,2}|
&\leq c_{3}\,\sum^N_{n=1}\sum^{\infty}_{j=N-n+1} |a_j|^{\alpha'}+c_{3}\,\sum^0_{n=-\infty}\sum^{N-n}_{j=1-n} |a_j|^{\alpha'}\\
&\leq c_{4}\,\sum^N_{n=1}\sum^{\infty}_{j=N-n+1} j^{-\delta}+c_{4}\,\sum^{N}_{n=0}\sum^{N+n}_{j=1+n} j^{-\delta}+c_{4}\,\sum^{\infty}_{n=N+1}\sum^{N+n}_{j=1+n} j^{-\delta}\\
&\leq c_{5}\,\big(1+\sum^{N-1}_{n=1} (N-n)^{1-\delta}\big)+c_{5}\,\big(1+\sum^{N}_{n=1}  n^{1-\delta}\big)+c_{5}\,\sum^{\infty}_{n=N+1}\big(n^{1-\delta}-(N+n)^{1-\delta}\big)\\
&\leq c_{6}\,N^{2-\delta}+c_{6}\,\int^{\infty}_N (x^{1-\delta}-(N+x)^{1-\delta})\, dx\\
&\leq c_{7}\,N^{2-\delta}+c_{7}\,N^{2-\delta}\int^{\infty}_1(x^{1-\delta}-(1+x)^{1-\delta})\, dx\\
&\leq c_{8}\,N^{2-\delta},
\end{align*}
where %in the last third and last inequalities we use the facts that 
%\[
%f_N(x):=x^{1-\delta}-(N+x)^{1-\delta}
%\] 
%is decreasing on $[N,\infty)$ 
in the last inequality we use the fact that $\delta\in (1,2) $ and 
\[
0<x^{1-\delta}-(1+x)^{1-\delta}\leq x^{-\delta}
\]
for all $x>0$.

Now we consider the case $\alpha\in(1,2)$. Recall that $\alpha'\in(1,\alpha)$ and $\beta'\in(1,\beta)$ with $\alpha'\beta'>1$. Then, by inequality (\ref{Ksum}) and von Bahr-Esseen inequality in \cite{vBE}, we could obtain that
\begin{align*}
\mathbb{E}|R_{N,1}|^{\alpha'}
&\leq 2 \sum^N_{n=1}\mathbb{E}\Big|\sum^{\infty}_{j=N-n+1}K^{\infty}_{j,n}\Big|^{\alpha'}\\
&\leq c_9\, \left(1+\sum^{N-1}_{n=1} \mathbb{E}\left| (N-n)^{1-\beta'}|\varepsilon_n|+(N-n)^{1-\alpha'\beta'}\right|^{\alpha'}\right)\\
&\leq c_{10}\, N^{1+\alpha'(1-\beta')}
\end{align*}
and
\begin{align*}
\mathbb{E}|R_{N,2}|^{\alpha'}
%&=\mathbb{E}\left(\liminf_{M\to \infty}\Big|\sum^0_{n=-M}\sum^{N-n}_{j=1-n}K^{\infty}_{j,n}\Big|^{\alpha'}\right)\\
%&\leq \liminf_{M\to \infty}\mathbb{E}\left|\sum^0_{n=-M}\sum^{N-n}_{j=1-n}K^{\infty}_{j,n}\right|^{\alpha'}\\
%&\leq 2 \liminf_{M\to \infty} \sum^0_{n=-M}\mathbb{E}\left|\sum^{N-n}_{j=1-n}K^{\infty}_{j,n}\right|^{\alpha'}\\
&\leq 2 \sum^0_{n=-\infty}\mathbb{E}\Big|\sum^{N-n}_{j=1-n}K^{\infty}_{j,n}\Big|^{\alpha'}\\
&\leq c_{11} \sum^{\infty}_{n=0}\Bigg[\bigg(\sum^{N+n}_{j=1+n}|a_j|\bigg)^{\alpha'} +\bigg(\sum^{N+n}_{j=1+n}|a_j|^{\alpha'}\bigg)^{\alpha'}\Bigg]\\
&\leq c_{12} \sum^{N}_{n=0} \bigg(\sum^{N+n}_{j=1+n}|a_j|\bigg)^{\alpha'}+c_{12} \sum^{\infty}_{n=N+1}\bigg(\sum^{N+n}_{j=1+n}|a_j|\bigg)^{\alpha'}\\
&\leq c_{13}\, \Big(1+\sum^{N}_{n=1} n^{\alpha'(1-\beta')}\Big)+c_{13}\sum^{\infty}_{n=N+1}\Big(n^{1-\beta'}-(N+n)^{1-\beta'}\Big)^{\alpha'}\\
&\leq c_{14}\, N^{1+\alpha'(1-\beta')},
\end{align*}
where in the last inequality we use similar arguments as in obtaining (\ref{sumint}).

Therefore, $\E |T_{N}-\mathcal{T}_{N}|^{\alpha'}$ is less than a constant multiple of $N^{1+\alpha'(1-\beta')}$ for the case $\alpha\in(1,2)$. This completes the proof.
\end{proof}

For each $x\in\R$, let 
\begin{align} \label{etakx}
\eta_K(x)=\sum\limits^{\infty}_{j=1} (K_{\infty}(a_j x)-\E K_{\infty}(a_j \varepsilon_1)).
\end{align}
The convergence of the infinite series in (\ref{etakx}) follows easily from (\ref{lipschitz}) below. Recall the definition of $\mathcal{T}_N$ in (\ref{taun}). Then $\mathcal{T}_N=\sum\limits^{N}_{n=1}  \eta_K(\varepsilon_n)$ and is a finite sum of i.i.d. random variables $\eta_K(\varepsilon_n)$.  

Next we show the asymptotic behavior of $\mathcal{T}_N$.
\begin{proposition} \label{prop52}  Under the assumptions of Theorem \ref{thm2},
\[
\frac{1}{\left(\frac{(\ell^{\leftarrow}_{\beta})^{\alpha}}{h\circ \ell^{\leftarrow}_{\beta}}\right)^{\leftarrow}(N)} \mathcal{T}_N \overset{\mathcal{L}}{\longrightarrow}  (\gamma_2+\gamma_1)^{\frac{1}{\alpha\beta}}\bigg(\frac{\gamma_2-\gamma_1}{(\alpha\beta-1)(\gamma_2+\gamma_1)}+\mathbf{Z}^{\alpha\beta}\bigg)
\]
as $N$ tends to infinity, where $\mathbf{Z}^{\alpha\beta}$ is an $\alpha\beta$-stable random variable whose characteristics function has the form
\begin{align} \label{Z}
\mathbb{E} e^{\iota u \mathbf{Z}^{\alpha\beta}} \nonumber
&=\exp\bigg(\iota u\frac{\gamma_2-\gamma_1}{\gamma_2+\gamma_1} \alpha\beta \left[\int^{\infty}_1\frac{x^{-\alpha\beta}}{x^2+1}dx-\int^1_0\frac{x^{2-\alpha\beta}}{x^2+1}dx\right]\\ 
&\qquad\qquad\qquad+\frac{\gamma_2}{\gamma_2+\gamma_1}\alpha\beta\int^{\infty}_0(e^{\iota u x}-1-\frac{\iota ux}{x^2+1})\frac{dx}{x^{\alpha\beta+1}} \nonumber \\
&\qquad\qquad\qquad\qquad+\frac{\gamma_1}{\gamma_2+\gamma_1}\alpha\beta\int^{0}_{-\infty}(e^{\iota u x}-1-\frac{\iota ux}{x^2+1})\frac{dx}{|x|^{\alpha\beta+1}}\bigg).
\end{align}
\end{proposition}

\begin{proof}  We divide the proof into several steps.

\noindent
{\bf Step 1} We first show that 
\begin{align} \label{etak}
\lim\limits_{x\to \pm\infty} \frac{1}{|x|^{\frac{1}{\beta}} \ell^{\frac{1}{\beta}}(|x|^{\frac{1}{\beta}})} \eta_K(x)=\int^{\infty}_0 \big(K_{\infty}(\pm t^{-\beta})-K_{\infty}(0)\big)dt:=C^{\pm}_K.
\end{align}
Using similar arguments as in obtaining (\ref{Ksum}), we can show that $K_{\infty}(a_j x)-\E K_{\infty}(a_j \varepsilon_1)$ is a Lipschitz continuous function on $\R$ and
\begin{align} \label{lipschitz}
|K_{\infty}(a_j x)-\E K_{\infty}(a_j \varepsilon_1)|\leq c_{1}\, (|a_j||x|+|a_j|^{\alpha'} )
\end{align}
for any $\alpha'\in(0,\alpha)$ with $\alpha'\beta>1$. So the series in the definition of $\eta_K(x)$ in (\ref{etakx}) converges absolutely for each $x\in\R$ and defines a Lipschitz continuous function on $\R$. 

For $x>0$,
\begin{align*}
\frac{1}{x^{\frac{1}{\beta}} \ell^{\frac{1}{\beta}}(x^{\frac{1}{\beta}})}\eta_K(x)
&=\frac{1}{x^{\frac{1}{\beta}}\ell^{\frac{1}{\beta}}(x^{\frac{1}{\beta}})}\sum^{\infty}_{j=1}  (K_{\infty}(a_jx)-K_{\infty}(0))+\frac{1}{x^{\frac{1}{\beta}}\ell^{\frac{1}{\beta}}(x^{\frac{1}{\beta}})}\sum^{\infty}_{j=1}  (K_{\infty}(0)-\mathbb{E}K_{\infty}(a_j\varepsilon_1))\\
&:=\text{I}+\text{II}.
\end{align*}

By (\ref{lipschitz}), $\sum\limits^{\infty}_{j=1} |K_{\infty}(0)-\mathbb{E}K_{\infty}(a_j\varepsilon_1)|<\infty$ and thus $|\text{II}|$ is less than a constant multiple of $\frac{1}{x^{\frac{1}{\beta}}\ell^{\frac{1}{\beta}}(x^{\frac{1}{\beta}})}$. Let $M$ be a large enough natural number. Then the term $\text{I}$ can be decomposed as 
\begin{align*}
\text{I}=\text{I}_1+\text{I}_2+\text{I}_3,
\end{align*}
where 
\begin{align*}
\text{I}_1&=\frac{1}{x^{\frac{1}{\beta}}\ell^{\frac{1}{\beta}}(x^{\frac{1}{\beta}})   }\sum^{[x^{\frac{1}{\beta}}\ell^{\frac{1}{\beta}}(x^{\frac{1}{\beta}})   /M]-1}_{j=1}  (K_{\infty}(a_jx)-K_{\infty}(0)),\\
\text{I}_2&=\frac{1}{x^{\frac{1}{\beta}}\ell^{\frac{1}{\beta}}(x^{\frac{1}{\beta}})   }\sum^{[Mx^{\frac{1}{\beta}}\ell^{\frac{1}{\beta}}(x^{\frac{1}{\beta}})   ]}_{j=[x^{\frac{1}{\beta}}\ell^{\frac{1}{\beta}}(x^{\frac{1}{\beta}})   /M]}  (K_{\infty}(a_jx)-K_{\infty}(0))
\end{align*}
and
\begin{align*}
\text{I}_3&=\frac{1}{x^{\frac{1}{\beta}}\ell^{\frac{1}{\beta}}(x^{\frac{1}{\beta}})   }\sum^{\infty}_{j=[Mx^{\frac{1}{\beta}}\ell^{\frac{1}{\beta}}(x^{\frac{1}{\beta}}) ]+1}  (K_{\infty}(a_jx)-K_{\infty}(0)).
\end{align*}
Recall that $K_{\infty}(x)=\mathbb{E} K(X_1+x)$ for each $x\in\R$.  By assumptions  ({\bf A1}) and ({\bf A2}), $K_{\infty}$ is bounded. So $|\text{I}_1|$ is less than a constant multiple of $\frac{1}{M}$.  

Note that $K_{\infty}$ is Lipschitz continuous and bounded, $a_j\sim j^{-\beta}\ell(j)$ and $\lim\limits_{x\to\infty} \frac{\ell(x\ell^{\frac{1}{\beta}}(x))}{\ell(x)}=1$. Then, by dominated convergence theorem,
\begin{align*}
\lim\limits_{x\to\infty}\text{I}_2
&=\lim\limits_{x\to\infty}\frac{1}{x^{\frac{1}{\beta}}\ell^{\frac{1}{\beta}}(x^{\frac{1}{\beta}})   }\sum^{[M x^{\frac{1}{\beta}}\ell^{\frac{1}{\beta}}(x^{\frac{1}{\beta}})   ]}_{j=[x^{\frac{1}{\beta}}\ell^{\frac{1}{\beta}}(x^{\frac{1}{\beta}})   /M]}  (K_{\infty}(j^{-\beta}\ell(j)x)-K_{\infty}(0))\\
&=\lim\limits_{x\to\infty}\frac{1}{x^{\frac{1}{\beta}}\ell^{\frac{1}{\beta}}(x^{\frac{1}{\beta}})   }\sum^{[M x^{\frac{1}{\beta}}\ell^{\frac{1}{\beta}}(x^{\frac{1}{\beta}})   ]}_{j=[x^{\frac{1}{\beta}}\ell^{\frac{1}{\beta}}(x^{\frac{1}{\beta}})   /M]}  (K_{\infty}(j^{-\beta}\ell(x^{\frac{1}{\beta}})x)-K_{\infty}(0))\\
&=\int^{M}_{\frac{1}{M}}\big(K_{\infty}(t^{-\beta})-K_{\infty}(0)\big)dt,
\end{align*}
where in the last equality follows from the definition of the integral.  

As for the term $\text{I}_3$, by Lipschitz continuity of $K_{\infty}$ and Theorem 2.1(b) in \cite{R},
\begin{align} \label{I3}
\limsup_{x\to\infty}|\text{I}_3|
%&=\limsup_{x\to\infty} \frac{1}{2\pi}\frac{1}{x^{\frac{1}{\beta}}\ell^{\frac{1}{\beta}}(x^{\frac{1}{\beta}})} \Bigg|\int_{\mathbb{R}} \widehat{K}(u)\phi(-u) \sum^{\infty}_{j=[Mx^{\frac{1}{\beta}}\ell^{\frac{1}{\beta}}(x^{\frac{1}{\beta}})]+1}  (e^{-\iota  u a_jx}-1)du \Bigg| \nonumber\\
&\leq c_2\limsup_{x\to\infty} \frac{1}{x^{\frac{1}{\beta}}\ell^{\frac{1}{\beta}}(x^{\frac{1}{\beta}})   } \sum^{\infty}_{j=[M x^{\frac{1}{\beta}}\ell^{\frac{1}{\beta}}(x^{\frac{1}{\beta}})]+1}  |a_j| x \nonumber\\
&\leq c_3\, \limsup_{x\to\infty} x^{-\frac{1}{\beta}}\ell^{-\frac{1}{\beta}}(x^{\frac{1}{\beta}})    (M x^{\frac{1}{\beta}}\ell^{\frac{1}{\beta}}(x^{\frac{1}{\beta}})   )^{1-\beta}\ell(Mx^{\frac{1}{\beta}}\ell^{\frac{1}{\beta}}(x^{\frac{1}{\beta}})) x \nonumber \\
&=c_3\, M^{1-\beta},
\end{align} 
where in the last equality we use the definition of slowly varying function and $\lim\limits_{x\to\infty} \frac{\ell(x\ell^{\frac{1}{\beta}}(x))}{\ell(x)}=1$.

Recall that $\beta>1$. Therefore, according to the above estimates for $\text{I}_1$, $\text{I}_2$, $\text{I}_3$ and $\text{II}$,  letting $M\to\infty$ gives 
\begin{align*}
\lim\limits_{x\to\infty} \frac{1}{x^{\frac{1}{\beta}}\ell^{\frac{1}{\beta}}(x^{\frac{1}{\beta}})}\eta_K(x)=\int^{\infty}_0 \big(K_{\infty}(t^{-\beta})-K_{\infty}(0)\big)dt.
\end{align*}
Similarly,
\[
\lim\limits_{x\to -\infty} \frac{1}{|x|^{\frac{1}{\beta}}\ell^{\frac{1}{\beta}}(|x|^{\frac{1}{\beta}})} \eta_K(x)=\int^{\infty}_0 \big(K_{\infty}(-t^{-\beta})-K_{\infty}(0)\big) dt.
\]

\noindent
{\bf Step 2} We next show that $\eta_K(\varepsilon_1)$ is in the domain of attraction of an $\alpha\beta$-stable law. That is,
\begin{align} \label{attraction}
\lim\limits_{x\to\infty} \frac{(\ell^{\leftarrow}_{\beta}(x))^{\alpha}}{h(\ell^{\leftarrow}_{\beta}(x))}\P\big(\eta_K(\varepsilon_1)>x\big)=\gamma_2\quad\text{and}\quad \lim\limits_{x\to\infty} \frac{(\ell^{\leftarrow}_{\beta}(x))^{\alpha}}{h(\ell^{\leftarrow}_{\beta}(x))}\P\big(\eta_K(\varepsilon_1)<-x\big)=\gamma_1,
\end{align}
where $\ell^{\leftarrow}_{\beta}(x)=\inf\{s>0: \ell_{\beta}(s)\geq x\}$ with $\ell_{\beta}(x)=x^{\frac{1}{\beta}}\ell^{\frac{1}{\beta}}(x^{\frac{1}{\beta}})$ defined in (\ref{lbx}), 
\[
\gamma_2=\sigma_2 (C^+_K)^{\alpha\beta} 1_{\{C^+_K>0\}}+\sigma_1 (C^-_K)^{\alpha\beta} 1_{\{C^-_K>0\}}\;\; \text{and}\;\; \gamma_1=\sigma_2 |C^+_K|^{\alpha\beta} 1_{\{C^+_K<0\}}+\sigma_1 |C^-_K|^{\alpha\beta} 1_{\{C^-_K<0\}}.
\] 
Proposition 2.6 in \cite{R} implies that $\frac{(\ell^{\leftarrow}_{\beta}(x))^{\alpha}}{h(\ell^{\leftarrow}_{\beta}(x))}$ is a regularly varying function at $\infty$ with index $\alpha\beta$. 

By assumption ({\bf A3}), there exists $A>0$ such that the regularly varying function $\ell_{\beta}(x)$ is continuous and strictly increasing on the interval $(A,\infty)$. So, for any $\delta\in (0,1)$, by (\ref{etak}), we could easily get 
\begin{align*}
&\limsup\limits_{x\to\infty} \frac{(\ell^{\leftarrow}_{\beta}(x))^{\alpha}}{h(\ell^{\leftarrow}_{\beta}(x))}\P\big(\eta_K(\varepsilon_1)>x\big)\\
&\qquad\leq 1_{\{C^+_K>0\}}\limsup\limits_{x\to\infty} \frac{(\ell^{\leftarrow}_{\beta}(x))^{\alpha}}{h(\ell^{\leftarrow}_{\beta}(x))}\P\Big(\ell_{\beta}(\varepsilon_1)>\frac{\ell_{\beta}(\varepsilon_1)}{\eta_K(\varepsilon_1)}x, \varepsilon_1>0\Big)\\
&\qquad\qquad\qquad+1_{\{C^-_K>0\}}\limsup\limits_{x\to\infty} \frac{(\ell^{\leftarrow}_{\beta}(x))^{\alpha}}{h(\ell^{\leftarrow}_{\beta}(x))}\P\Big(\ell_{\beta}(|\varepsilon_1|)>\frac{\ell_{\beta}(|\varepsilon_1|)}{\eta_K(\varepsilon_1)}x, \varepsilon_1<0\Big)\\
&\qquad\leq 1_{\{C^+_K>0\}}\limsup\limits_{x\to\infty} \frac{(\ell^{\leftarrow}_{\beta}(x))^{\alpha}}{h(\ell^{\leftarrow}_{\beta}(x))}\P\Big(\ell_{\beta}(\varepsilon_1)>\frac{1-\delta}{C^+_K}x, \varepsilon_1>0\Big)\\
&\qquad\qquad\qquad+1_{\{C^-_K>0\}}\limsup\limits_{x\to\infty} \frac{(\ell^{\leftarrow}_{\beta}(x))^{\alpha}}{h(\ell^{\leftarrow}_{\beta}(x))}\P\Big(\ell_{\beta}(|\varepsilon_1|)>\frac{1-\delta}{C^-_K} x, \varepsilon_1<0\Big)\\
&\qquad\leq 1_{\{C^+_K>0\}}\limsup\limits_{x\to\infty} \frac{(\ell^{\leftarrow}_{\beta}(x))^{\alpha}}{h(\ell^{\leftarrow}_{\beta}(x))}\P\Big(\varepsilon_1>\ell^{\leftarrow}_{\beta}(\frac{1-\delta}{C^+_K}x)\Big)\\
&\qquad\qquad\qquad+1_{\{C^-_K>0\}}\limsup\limits_{x\to\infty} \frac{(\ell^{\leftarrow}_{\beta}(x))^{\alpha}}{h(\ell^{\leftarrow}_{\beta}(x))}\P\Big(\varepsilon_1<-\ell^{\leftarrow}_{\beta}(\frac{1-\delta}{C^-_K} x)\Big).
\end{align*}
Now, by (\ref{h}) and Proposition 2.6 in \cite{R},
\begin{align*}
\limsup\limits_{x\to\infty} \frac{(\ell^{\leftarrow}_{\beta}(x))^{\alpha}}{h(\ell^{\leftarrow}_{\beta}(x))}\P\big(\eta_K(\varepsilon_1)>x\big)
&\leq (\sigma_2+\delta)\, \Big(\frac{C^+_K}{1-\delta}\Big)^{\alpha\beta} 1_{\{C^+_K>0\}}+(\sigma_1+\delta)\, \Big(\frac{C^-_K}{1-\delta}\Big)^{\alpha\beta} 1_{\{C^-_K>0\}}.
\end{align*}
Letting $\delta\downarrow 0$ gives 
\begin{align*}
&\limsup\limits_{x\to\infty} \frac{(\ell^{\leftarrow}_{\beta}(x))^{\alpha}}{h(\ell^{\leftarrow}_{\beta}(x))}\P\big(\eta_K(\varepsilon_1)>x\big)\leq \sigma_2 (C^+_K)^{\alpha\beta} 1_{\{C^+_K>0\}}+\sigma_1 (C^-_K)^{\alpha\beta} 1_{\{C^-_K>0\}}.
\end{align*}
Similarly, 
\begin{align*}
&\liminf\limits_{x\to\infty} \frac{(\ell^{\leftarrow}_{\beta}(x))^{\alpha}}{h(\ell^{\leftarrow}_{\beta}(x))}\P\big(\eta_K(\varepsilon_1)>x\big)\geq (\sigma_2-\delta) \Big(\frac{C^+_K}{1+\delta}\Big)^{\alpha\beta} 1_{\{C^+_K>0\}}+(\sigma_1-\delta) \Big(\frac{C^-_K}{1+\delta}\Big)^{\alpha\beta} 1_{\{C^-_K>0\}}.
\end{align*}
Letting $\delta\downarrow 0$ gives 
\begin{align*}
&\liminf\limits_{x\to\infty} \frac{(\ell^{\leftarrow}_{\beta}(x))^{\alpha}}{h(\ell^{\leftarrow}_{\beta}(x))}\P\big(\eta_K(\varepsilon_1)>x\big)\geq \sigma_2 (C^+_K)^{\alpha\beta} 1_{\{C^+_K>0\}}+\sigma_1 (C^-_K)^{\alpha\beta} 1_{\{C^-_K>0\}}.
\end{align*}
Therefore, we obtain the first limit in (\ref{attraction}). The second one can be obtained similarly.

\noindent
{\bf Step 3}  Let $a_N=\inf\{x>0: \mathbb{P}(|\eta_K(\varepsilon_1)|>x)\leq \frac{1}{N}\}$. Note that
\begin{align*}
\E[\eta_K(\varepsilon_1) 1_{\{|\eta_K(\varepsilon_1)|> a_N\}}]
&=\int^{\infty}_{a_N}\P(\eta_K(\varepsilon_1))>x)dx-\int^{-a_N}_{-\infty}\P(\eta_K(\varepsilon_1))<-x)dx.
\end{align*}
Recall that $\alpha\beta>1$. Then, by Theorem 2.1(a) in \cite{R}, (\ref{attraction}) and  (3.7.6)-(3.7.7) in \cite{D},
\begin{align*}
\lim_{N\to\infty} \frac{N\E[\eta_K(\varepsilon_1) 1_{\{|\eta_K(\varepsilon_1)|> a_N\}}]}{a_N}
&=\frac{1}{\alpha\beta-1}   \lim_{N\to\infty}  N  \Big( \P(\eta_K(\varepsilon_1))>a_N)-\P(\eta_K(\varepsilon_1))<-a_N)\Big)  \\
&=\frac{\gamma_2-\gamma_1}{(\alpha\beta-1)(\gamma_2+\gamma_1)}.
\end{align*}
Note that $a_N$ can be rewritten as $a_N=\inf\{x>0: \frac{1}{\mathbb{P}(|\eta_K(\varepsilon_1)|>x)}\geq N\}$. Then, by (\ref{attraction}) and Proposition 2.6 in \cite{R},
\begin{align*}
\lim_{N\to\infty} \frac{a_N}{\left(\frac{(\ell^{\leftarrow}_{\beta})^{\alpha}}{h\circ (\ell^{\leftarrow}_{\beta})}\right)^{\leftarrow}(N)}=(\gamma_2+\gamma_1)^{\frac{1}{\alpha\beta}}.
\end{align*}
Now,  by Theorem 3.7.2  and (3.7.10) in \cite{D}, we get the desired result.
\end{proof}

\medskip
\noindent
{\bf Proof of Theorem \ref{thm2}}:  Using similar arguments as in the proof of Theorem 9.32 in \cite{B},  (\ref{Z}) can be rewritten as 
\begin{align*}
\mathbb{E} e^{\iota u \mathbf{Z}^{\alpha\beta}}&=\exp\Bigg(\iota u\left[\frac{\gamma_2-\gamma_1}{\gamma_2+\gamma_1}\alpha\beta\left(\int^{\infty}_1\frac{x^{-\alpha\beta}}{x^2+1}dx-\int^1_0\frac{x^{2-\alpha\beta}}{x^2+1}dx+\frac{1}{\alpha\beta}\int^{\infty}_0\frac{x^{2-\alpha\beta}(x^{2}+3)}{(x^2+1)^2} dx \right)\right]\\
&\qquad\qquad\qquad\qquad-\left(\int^{\infty}_0  \frac{\sin x}{x^{\alpha\beta}} dx\right)|u|^{\alpha\beta}\left[1-\iota \frac{\gamma_2-\gamma_1}{\gamma_2+\gamma_1} \sgn(u)\tan(\frac{\pi\alpha\beta}{2})\right]\Bigg).
\end{align*}

Recall that $\mathcal{T}_{N}$ is a finite sum of i.i.d. random variables in the domain of attraction of an $\alpha\beta$-stable law.  So the process $\{\mathcal{T}_{[Nt]}:\, t\geq 0\}$ has independent increments. Using characteristic functions and Proposition \ref{prop52}, we could easily show that  
\begin{align} \label{taufdd}
\Bigg\{\frac{1}{\Big(\frac{(\ell^{\leftarrow}_{\beta})^{\alpha}}{h\circ (\ell^{\leftarrow}_{\beta})}\Big)^{\leftarrow}(N)} \mathcal{T}_{[Nt]}: \; t\geq 0 \Bigg\} \overset{f.d.d.}{\longrightarrow}  \left\{(\gamma_2+\gamma_1)^{\frac{1}{\alpha\beta}}\big(\overline{c}\, t^{\frac{1}{\alpha\beta}}+\mathcal{Z}^{\alpha\beta}_t \big): \; t\geq 0\right\}
\end{align}
as $N$ tends to infinity, where ``\,$\overset{\, f.d.d.\, }{\longrightarrow}$\,'' denotes the convergence of finite-dimensional distributions.

For any $t>0$, by Proposition \ref{prop} and Lemma \ref{lm51}, we can easily obtain that 
\[
\frac{1}{\left(\frac{(\ell^{\leftarrow}_{\beta})^{\alpha}}{h\circ (\ell^{\leftarrow}_{\beta})}\right)^{\leftarrow}(N)} \big(S_{[Nt]}-\mathcal{T}_{[Nt]}\big)
\] 
converges in probability to $0$ as $N$ tends to infinity. Therefore, the desired convergence of finite-dimensional distributions follows easily from (\ref{taufdd}). This completes the proof.

\subsection{Proof of Theorem \ref{thm3}}

In this subsection, we give the proof of Theorem \ref{thm3}. Recall that 
\begin{align*}
T_{N}=\sum^{N}_{n=1} \sum^{\infty}_{j=1} \big(K_{\infty}(a_j \varepsilon_{n-j})-\E K_{\infty}(a_j \varepsilon_{n-j})\big)\;\; \text{and}\;\;
\mathcal{T}_{N}=\sum^{N}_{n=1} \sum^{\infty}_{j=1} \big(K_{\infty}(a_j \varepsilon_{n})-\E K_{\infty}(a_j \varepsilon_{n}) \big).
\end{align*}
The convergence of the infinite series in $\mathcal{T}_{N}$ under the setup of Theorem \ref{thm3} follows easily from the inequality (\ref{tncon2}) below.

We first give the estimate for $\E |T_{N}-\mathcal{T}_{N}|$.

\begin{lemma} \label{lm53} Suppose that assumptions ({\bf A1})-({\bf A2}) hold, $\alpha\in (1,2)$, $\beta\in(\frac{1}{\alpha},1]$ and $\int_{\R} K(x)df(x)=0$. Then, for any $\alpha'\in (1,\alpha)$ and $\beta'\in (0,\beta)$ with   $\alpha'\beta'>1$, there is a positive constant $c_{\alpha', \beta', 5}>0$ such that
\[
\E |T_{N}-\mathcal{T}_{N}| \leq c_{\alpha',\beta', 5}\, N^{2-\alpha'\beta'}.
\]
\end{lemma}

\begin{proof} Note that 
\begin{align*}
\mathcal{T}_{N}-T_{N}=R_{N,1}-R_{N,2}
&:=\left(\sum^N_{n=1}\sum^{\infty}_{j=N-n+1}-\sum^0_{n=-\infty}\sum^{N-n}_{j=1-n}\right)\Big(K_{\infty}(a_j\varepsilon_n)-\E K_{\infty}(a_j\varepsilon_n)\Big)
\end{align*}
and
\begin{align*}
K_{\infty}(a_j\varepsilon_n)-\E K_{\infty}(a_j\varepsilon_n)
&=\frac{1}{2\pi}\int_{\R} \widehat{K}(u) \phi(-u)(e^{-\iota u a_j \varepsilon_{n}}-\phi_{\varepsilon}(-a_j u)) du
\end{align*}

Let $K^{\infty}_{j,n}=K_{\infty}(a_j\varepsilon_n)-\E K_{\infty}(a_j\varepsilon_n)$. Then, by assumptions  ({\bf A1})-({\bf A2}) and $\int_{\R} K(x)df(x)=0$, we can obtain that 
\[
\int_{\R} \widehat{K}(u)\phi(-u) u\, du=0
\] 
and thus
\begin{align} \label{tncon2}
|K^{\infty}_{j,n}|   \nonumber
&=\frac{1}{2\pi}\Big|\int_{\R} \widehat{K}(u)\phi(-u)(e^{-\iota u a_j \varepsilon_{n}}+\iota u a_j \varepsilon_{n}-\phi_{\varepsilon}(-a_j u)) du\Big|\\  \nonumber
&\leq \int_{\R} |\widehat{K}(u)| |\phi(u)|\big(|e^{-\iota u a_j \varepsilon_{n}}-1+\iota u a_j \varepsilon_{n}|+|1-\phi_{\varepsilon}(-a_j u)|\big) du\\  \nonumber
&\leq c_{1}\, \int_{\R} |\widehat{K}(u)| |\phi(u)|\big( |u a_j \varepsilon_{n}|^{\alpha'}+|a_j u|^{\alpha'}\big) du\\ 
&\leq c_{2}\, |a_j|^{\alpha'} (|\varepsilon_n|^{\alpha'}+1).
\end{align}
Then, using similar arguments as in the proof of Lemma \ref{lm51}, we can get the desired estimate.
%By Theorem 2.6.4 in \cite{IL}, $\mathbb{E} |\varepsilon_n|^{\alpha'}<\infty$. Let $\delta=\alpha'\beta'$. Clearly, $\delta\in (1,2)$. We can obtain that
%\begin{align*}
%\E |R_{N,1}|+\E |R_{N,2}|
%&\leq c_{3}\, \sum^N_{n=1}\sum^{\infty}_{j=N-n+1} |a_j|^{\alpha'}+c_{3}\, \sum^0_{n=-\infty}\sum^{N-n}_{j=1-n} |a_j|^{\alpha'}\\
%&\leq c_{4}\, \sum^N_{n=1}\sum^{\infty}_{j=N-n+1} j^{-\delta}+c_{4}\, \sum^{N}_{n=0}\sum^{N+n}_{j=1+n} j^{-\delta}+c_{15}\sum^{\infty}_{n=N+1}\sum^{N+n}_{j=1+n} j^{-\delta}\\
%&\leq c_{5}\, (1+\sum^{N-1}_{n=1} (N-n)^{1-\delta})+c_{5}\, (1+\sum^{N}_{n=1}  n^{1-\delta})+c_{5}\, \sum^{\infty}_{n=N+1}(n^{1-\delta}-(N+n)^{1-\delta})\\
%&\leq c_{6}\, N^{2-\delta}+c_{6}\, \int^{\infty}_N (x^{1-\delta}-(N+x)^{1-\delta})\, dx\\
%&\leq c_{7}\, N^{2-\delta}+c_{7}\, N^{2-\delta}\int^{\infty}_1(x^{1-\delta}-(1+x)^{1-\delta})\, dx\\
%&\leq c_{8}\, N^{2-\delta},
%\end{align*}
%where in the last third and last inequalities we used the facts that
%\[
%f_N(x)=x^{1-\delta}-(N+x)^{1-\delta}
%\] 
%is decreasing on $[N,\infty)$ and
%\[
%0\leq f_1(x)=x^{1-\delta}-(1+x)^{1-\delta}\leq x^{-\delta}
%\]
%on $[1,\infty)$, respectively.  
\end{proof}

Next we show the asymptotic behavior of $\mathcal{T}_N$.
\begin{proposition} \label{prop53} Under the assumptions of Theorem \ref{thm3},
\[
\frac{1}{\left(\frac{(\ell^{\leftarrow}_{\beta})^{\alpha}}{h\circ \ell^{\leftarrow}_{\beta}}\right)^{\leftarrow}(N)} \mathcal{T}_N \overset{\mathcal{L}}{\longrightarrow}  (\gamma_2+\gamma_1)^{\frac{1}{\alpha\beta}}\left(\frac{\gamma_2-\gamma_1}{(\alpha\beta-1)(\gamma_2+\gamma_1)}+\mathbf{Z}^{\alpha\beta}\right)
\]
as $N$ tends to infinity, where $\mathbf{Z}^{\alpha\beta}$ is the $\alpha\beta$-stable random variable with characteristics function given in (\ref{Z}). 
\end{proposition}

\begin{proof} By assumptions  ({\bf A1})-({\bf A2}), Plancherel formula and $\int_{\R} K(x)df(x)=0$,
\begin{align*}
\int^{\infty}_1 \left| K_{\infty}(t^{-\beta})-K_{\infty}(0)\right| dt
&=\frac{1}{2\pi}\int^{\infty}_1 \left| \int_{\R} \widehat{K}(u)\phi(-u)(e^{-\iota u t^{-\beta}}-1)du \right| dt\\
&=\frac{1}{2\pi} \int^{\infty}_1 \left|  \int_{\R} \widehat{K}(u)\phi(-u)(e^{-\iota u t^{-\beta}}-1+\iota u t^{-\beta})du \right| dt\\
&\leq c_1 \int^{\infty}_1 \int_{\R} \frac{1}{1+u^4}\frac{u^2}{t^{2\beta}}du\, dt\\
&<\infty,
\end{align*}
where we use the fact that $2\beta>\frac{2}{\alpha}>1$ in the last inequality. 

Note that the function $K_{\infty}$ is bounded and Lipschitz continuous. Hence $\int^{\infty}_0 \big(K_{\infty}(t^{-\beta})-K_{\infty}(0)\big)dt$ is finite for $\beta\in (\frac{1}{\alpha}, 1]$. Similarly, $\int^{\infty}_0 \big(K_{\infty}(-t^{-\beta})-K_{\infty}(0)\big)dt$ is finite for $\beta\in (\frac{1}{\alpha}, 1]$.  Now the desired result follows from almost the same arguments as in the proof of Proposition \ref{prop52} with (\ref{I3}) replaced by
\begin{align*}
\limsup_{x\to\infty}|\text{I}_3|
&=\limsup_{x\to\infty} \frac{1}{2\pi}\frac{1}{x^{\frac{1}{\beta}}\ell^{\frac{1}{\beta}}(x^{\frac{1}{\beta}})   } \bigg|\int_{\mathbb{R}} \widehat{K}(u) \phi(-u) \sum^{\infty}_{j=[Mx^{\frac{1}{\beta}}\ell^{\frac{1}{\beta}}(x^{\frac{1}{\beta}})]+1}  (e^{-\iota  u a_jx}-1+\iota u a_j x)du \bigg|\\
&\leq c_2 \limsup_{x\to\infty} \frac{1}{x^{\frac{1}{\beta}}\ell^{\frac{1}{\beta}}(x^{\frac{1}{\beta}})} \sum^{\infty}_{j=[M x^{\frac{1}{\beta}}\ell^{\frac{1}{\beta}}(x^{\frac{1}{\beta}})]+1}  a^2_j x^2\\
&\leq c_3\, \limsup_{x\to\infty} x^{-\frac{1}{\beta}}\ell^{-\frac{1}{\beta}}(x^{\frac{1}{\beta}})    (M x^{\frac{1}{\beta}}\ell^{\frac{1}{\beta}}(x^{\frac{1}{\beta}})   )^{1-2\beta}\ell^2(Mx^{\frac{1}{\beta}}\ell^{\frac{1}{\beta}}(x^{\frac{1}{\beta}})) x^2\\
&=c_3\, M^{1-2\beta}.
\end{align*}
This completes the proof.
\end{proof}

\medskip
\noindent
{\bf Proof of Theorem \ref{thm3}}: This follows from Proposition \ref{prop}, Lemma \ref{lm53}, Proposition \ref{prop53} and similar arguments as in the proof of Theorem \ref{thm2}.

\section*{Acknowledgements} We would like to thank the editor and an anonymous referee for their valuable comments.

\bigskip

$\begin{array}{cc}
\begin{minipage}[t]{1\textwidth}
{\bf Hui Liu}

\medskip
School of Statistics, East China Normal University, Shanghai 200262, China \\
\texttt{lhui56@163.com}
\end{minipage}
\hfill
\end{array}$

\medskip

$\begin{array}{cc}
\begin{minipage}[t]{1\textwidth}
{\bf Yudan Xiong}

\medskip
School of Statistics, East China Normal University, Shanghai 200262, China \\
\texttt{xyd\_980107@163.com}
\end{minipage}
\hfill
\end{array}$

\medskip

$\begin{array}{cc}

\begin{minipage}[t]{1\textwidth}

{\bf Fangjun Xu}

\medskip

KLATASDS-MOE, School of Statistics, East China Normal University, Shanghai, 200062, China 

\medskip
NYU-ECNU Institute of Mathematical Sciences at NYU Shanghai, Shanghai, 200062, China\\
\texttt{fjxu@finance.ecnu.edu.cn, fangjunxu@gmail.com}

\end{minipage}

\hfill

\end{array}$

\end{document}